\documentclass[11pt]{article}
\usepackage{amsmath}
\usepackage{amsthm}
\usepackage{amssymb,amsfonts}
\usepackage{hyperref}
\usepackage{latexsym}
\usepackage{todonotes}
\usepackage{nicefrac}
\usepackage{epsfig}
\usepackage{stmaryrd}
\usepackage{setspace}
\usepackage{enumerate}
\usepackage[all]{xypic}
\usepackage{bbm,ifpdf,tikz}
\ifpdf
\usepackage{pdfsync}
\fi

\newtheorem{theorem}{Theorem}[section]

\newtheorem{lemma}[theorem]{Lemma}
\newtheorem{proposition}[theorem]{Proposition}
\newtheorem{corollary}[theorem]{Corollary}

\newtheorem{conjecture}[theorem]{Conjecture}

{
\theoremstyle{definition}

\newtheorem{example}[theorem]{Example}

\newtheorem{remark}[theorem]{Remark}
}

\newcommand{\becircled}{\mathaccent "7017}
\newcommand{\excise}[1]{}

\newcommand{\Proj}{\operatorname{Proj}}

\renewcommand{\dim}{\operatorname{dim}}

\newcommand{\rk}{\operatorname{rk}}
\newcommand{\crk}{\operatorname{crk}}

\renewcommand{\and}{\qquad\text{and}\qquad}

\newcommand{\Z}{\mathbb{Z}}

\newcommand{\N}{\mathbb{N}}

\newcommand{\C}{\mathbb{C}}

\newcommand{\Fq}{\mathbb{F}_q}
\newcommand{\Fqb}{\overline{\mathbb{F}}_q}
\newcommand{\Fqs}{\mathbb{F}_{q^s}}
\newcommand{\IC}{\operatorname{IC}}
\renewcommand{\H}{\operatorname{H}}
\newcommand{\IH}{\operatorname{IH}}
\newcommand{\IHX}{\IH^*(X; \Ql)}
\newcommand{\IHxX}{\IH^*_x(X; \Ql)}
\newcommand{\IHiX}{\IH^i(X; \Ql)}
\newcommand{\Fr}{\operatorname{Fr}}
\newcommand{\cI}{\mathcal{I}}
\newcommand{\Ql}{\overline{\mathbb{Q}}_\ell}
\renewcommand{\a}{\alpha}
\renewcommand{\b}{\beta}
\newcommand{\cA}{\mathcal{A}}
\newcommand{\la}{\lambda}
\newcommand{\IF}{\cI\smallsetminus F}
\newcommand{\bc}{\bar{c}}
\newcommand{\ep}{\varepsilon}

\newcommand{\nicktodo}{\todo[inline,color=green!20]}

\begin{document}
\spacing{1.2}
\noindent{\Large\bf The Kazhdan-Lusztig polynomial of a matroid}\\

\noindent{\bf Ben Elias}\\
Department of Mathematics, University of Oregon,
Eugene, OR 97403\\
belias@uoregon.edu\\

\noindent{\bf Nicholas Proudfoot}\footnote{Supported by NSF grant DMS-0950383.}\\
Department of Mathematics, University of Oregon,
Eugene, OR 97403\\
njp@uoregon.edu\\

\noindent{\bf Max Wakefield}\footnote{Supported by the Simons Foundation and the Office of Naval Research.}\\
Department of Mathematics, United States Naval Academy, Annapolis, MD 21402\\
wakefiel@usna.edu\\
{\small
\begin{quote}
\noindent {\em Abstract.}
We associate to every matroid $M$ a polynomial with integer coefficients, 
which we call the Kazhdan-Lusztig polynomial of $M$,
in analogy with Kazhdan-Lusztig polynomials in representation theory.
We conjecture that the coefficients are always non-negative, and we prove
this conjecture for representable matroids by interpreting our
polynomials as intersection cohomology Poincar\'e polynomials.  We also 
introduce a $q$-deformation of the M\"obius algebra of $M$, and use our polynomials
to define a special basis for this deformation, analogous to the canonical basis of the Hecke algebra.
We conjecture that the structure coefficients for multiplication in this special basis are non-negative,
and we verify this conjecture in numerous examples.
\end{quote} }

\section{Introduction}
Our goal is to develop Kazhdan-Lusztig theory for matroids in analogy with
the well-known theory for Coxeter groups.  In order to make this analogy clear, we begin
by summarizing the most relevant features of the usual theory.

Given a Coxeter group $W$ along with a pair of elements $y,w\in W$, Kazhdan and Lusztig \cite{KL79}
associated a polynomial $P_{x,y}(t)\in\Z[t]$, which is non-zero if and only if $x\leq y$ in the Bruhat order.
This polynomial has a number of different interpretations:

\begin{itemize}
\item {\bf Combinatorics:}  There is a purely combinatorial recursive definition of $P_{x,y}(t)$ 
in terms of more elementary polynomials, called $R$-polynomials.  
See \cite[Proposition 2]{Lusztig-Left}, as well as \cite[\S 5.5]{BB-Coxeter}
for a more recent account.
\item {\bf Geometry:}  If $W$ is a finite Weyl group, then $P_{x,y}(t)$ may be interpreted as the Poincar\'e
polynomial of a stalk of the intersection cohomology sheaf on a Schubert variety 
in the associated flag variety \cite{KL80}.
The Schubert variety is determined by $y$, and the point at which one takes the stalk is determined by $x$.
This proves that $P_{x,y}(t)$ has non-negative coefficients when $W$ is a finite Weyl group.
The non-negativity of the coefficients of $P_{x,y}(t)$ for arbitrary Coxeter groups was 
conjectured in \cite{KL79}, but was
only recently proved by Williamson and the first author \cite[1.2(1)]{EW14}.
\item {\bf Algebra:}  The polynomials $P_{x,y}(t)$ are the entries of the matrix relating the Kazhdan-Lusztig basis
(or canonical basis) 
to the standard basis of the Hecke algebra of $W$, a $q$-deformation of the group algebra $\C[W]$.
When $W$ is a finite Weyl group, Kazhdan and Lusztig showed that the structure coefficients for multiplication
in the Kazhdan-Lusztig basis are polynomials with non-negative coefficients.  For general Coxeter groups, this
is proved in \cite[1.2(2)]{EW14}.
\end{itemize}

In our analogy, the Coxeter group $W$ is replaced by a matroid $M$, and the elements $x,y\in W$
are replaced by flats $F$ and $G$ of $M$.  We only define a single polynomial
$P_M(t)$ for each matroid, but one may associate to a pair $F\leq G$ the polynomial $P_{M^F_G}(t)$,
where $M^F_G$ is the matroid whose lattice of flats is isomorphic to the interval\footnote{We 
note that this shortcut has no analogue in ordinary Kazhdan-Lusztig theory, since the interval $[x,y]$
is not in general isomorphic to the Bruhat poset of some other Coxeter group.  Furthermore, it is still
an open question whether or not $P_{x,y}(t)$ is determined by the isomorphism type of the interval $[x,y]$.} $[F,G]$.
The role of the $R$-polynomial is played by the characteristic polynomial of the matroid.
The analogue of being a finite Weyl group is being a representable matroid;
that is, the matroid $M_\cA$ associated to a collection $\cA$ of vectors in a vector space.  The analogue of
a Schubert variety is the reciprocal plane $X_\cA$, also known as the spectrum of the Orlik-Terao
algebra of $\cA$.  The analogue of the group algebra $\C[W]$ is the M\"obius algebra $E(M)$;
we introduce a $q$-defomation $E_q(M)$ of this algebra which plays the role of the Hecke algebra.  

All of these analogies
may be summarized as follows:

\begin{itemize}
\item {\bf Combinatorics:}  We give a recursive definition of the polynomial $P_M(t)$
in terms of the characteristic polynomial of a matroid (Theorem \ref{KL-exists}),
and we conjecture that the coefficients are non-negative (Conjecture \ref{positivity}).
\item {\bf Geometry:}  If $M$ is representable over a finite field, we show that $P_M(t)$ is equal to the 
$\ell$-adic \'etale intersection cohomology Poincar\'e polynomial of the reciprocal plane\footnote{The
reciprocal plane is a cone, so we could equivalently say that it is the Poincar\'e polynomial
of the stalk of the intersection cohomology sheaf at the cone point.}  (Theorem \ref{comb=geom}).
Any matroid that is representable over some field is representable over a finite field, thus
we obtain a proof of Conjecture \ref{positivity} for all representable matroids (Corollary \ref{nonneg-cor}).
\item {\bf Algebra:}  We use the polynomials $P_{M^F_G}(t)$ to define the Kazhdan-Lusztig basis of the $q$-deformed
M\"obius algebra $E_q(M)$.  We conjecture that the structure constants for multiplication in this basis
are polynomials in $q$ with non-negative coefficients (Conjecture \ref{structure coefs}), and we verify this
conjecture in a number of cases.
\end{itemize}

\begin{remark} Despite these parallels, the behavior of the polynomials for matroids differs drastically 
from the behavior of ordinary Kazhdan-Lusztig polynomials for Coxeter groups. 
In particular, one does not recover the classical Kazhdan-Lusztig polynomials for the Coxeter group $S_n$
from the braid matroid.  Polo \cite{Polo} has shown that any polynomial with non-negative coefficients
and constant term 1 appears as a Kazhdan-Lusztig polynomial associated to some symmetric group, 
while Kazhdan-Lusztig polynomials of matroids are far more restrictive (see Proposition \ref{lin-non-neg}).
\end{remark}

The original work of Kazhdan and Lusztig begins with an algebraic question (How can we find a basis
for the Hecke algebra with certain nice properties?), which led them to both the combinatorics and the
geometry. In our work, we began with a geometric question (What is the intersection cohomology
of the reciprocal plane?), which led us naturally to the combinatorics.  The algebraic facet of our
work is somewhat more speculative and {\em ad hoc}, representing an attempt to trace backward the route of Kazhdan and Lusztig. 

There is no known convolution product in the geometry of the reciprocal plane which would account for the $q$-deformed M\"obius algebra 
$E_q(M)$, as the convolution product on flag varieties produces the Hecke algebra. Unlike in the Coxeter setting, 
the Kazhdan-Lusztig basis of $E_q(M)$ currently has no intrinsic definition, and the theory of this basis is far less
satisfactory. For example, the basis is cellular, but in a trivial way: the cells are all one-dimensional. The identity is not an element of the basis.

\begin{remark} When $W$ is a finite Weyl group, yet another important interpretation of $P_{x,y}(t)$ is that it records the multiplicity space of a simple module in a Verma module in the
graded lift of Berstein-Gelfand-Gelfand category $\mathcal{O}$ \cite{BB,BryKash}. The analogous goal for matroids would be a categorification of the $q$-deformed M\"obius algebra $E_q(M)$,
or its regular representation. The M\"obius algebra $E(M)$ is categorified by a monoidal category of ``commuting" quiver representations \cite[Theorem 7]{Baclawski}, but we do not know how
to modify this category to produce a categorification of $E_q(M)$. \end{remark}

Having made these caveats, the observed phenomenon of positivity indicates that our Kazhdan-Lusztig basis does hold interest. There are numerous other ways one could have used the
Kazhdan-Lusztig polynomials of a matroid as a change of basis matrix, but the corresponding bases do not have positive structure coefficients. As seen in Remark \ref{needKLornotpositive},
positivity is a subtle question, and would fail if all the Kazhdan-Lusztig polynomials were trivial.\\

We now give a more detailed summary of the contents of the paper.  Section \ref{sec:combinatorics} (Combinatorics)
is dedicated to the combinatorial definition of $P_M(t)$ along with basic properties and examples.
In addition to our conjecture that the coefficients of $P_M(t)$ are non-negative (Conjecture
\ref{positivity}), we also conjecture that they form a log concave sequence (Conjecture \ref{log concave}).
We explicitly compute the coefficients of $t$ and $t^2$ in terms of the Whitney numbers of the lattice of flats of $M$
(Propositions \ref{linear term} and \ref{quadratic term}).
We prove non-negativity of the linear coefficient (Proposition \ref{lin-non-neg}),
and we give formulas for the quadratic and cubic term (Propositions \ref{quadratic term} and \ref{cubic term}), 
though even in these cases we cannot prove non-negativity (Remark \ref{half of life}).
We prove a product formula for direct sums 
(Proposition \ref{product}), which eliminates the possibility of ``cheap" counterexamples to Conjecture
\ref{positivity} (Remark \ref{cheap}).

We also study in detail the cases of uniform matroids and braid matroids.  For uniform matroids,
we provide an even more explicit computation of the polynomial up to the cubic term 
(Corollary \ref{uniform coefficients}). 
For the braid matroid $M_n$ corresponding to the complete graph on $n$ vertices, we explain
how to compute the coefficients of the Kazhdan-Lusztig polynomial using Stirling numbers.
In an appendix, written jointly with Ben Young, we give tables of Kazhdan-Lusztig polynomials
of uniform matroids and braid matroids of low rank.
The polynomials that we see are unfamiliar; in particular, they do not appear to be related
to any known matroid invariants.  
For both uniform matroids and braid matroids,
we express the defining recursion in terms of a generating function identity (Propositions \ref{uniform gf}
and \ref{braid gf}).\\

The purpose of Section \ref{sec:geometry} (Geometry)
is to prove that, if $M_\cA$ is the matroid associated to a vector arrangement
$\cA$ over a finite field, then the Kazhdan-Lusztig polynomial of $M_\cA$ coincides with the $\ell$-adic \'etale
intersection cohomology Poincar\'e polynomial of the reciprocal plane $X_\cA$ (Theorem \ref{comb=geom}).
The key ingredient to our proof is Theorem \ref{etale}, which says that, in an \'etale neighborhood of
any point, $X_\cA$ looks like the product of a vector space with a neighborhood
of the cone point in the reciprocal plane
of a certain smaller hyperplane arrangement.  This improves upon a result of Sanyal, Sturmfels, and Vinzant
\cite[Theorem 24]{SSV}, who prove the analogous statement on the level of tangent cones.

We conclude Section \ref{sec:geometry} with a digression in which we discuss a 
certain question of Li and Yong \cite{LiYong}.  Given a point on a variety, they compare two polynomials:
the local intersection cohomology Poincar\'e polynomial, and the numerator of the Hilbert series of the tangent cone.
They are interested in the case of Schubert varieties, where the first polynomial is a Kazhdan-Lusztig polynomial.
We consider the case of reciprocal planes, where the first polynomial is the Kazhdan-Lusztig polynomial of a matroid
and the second polynomial is the $h$-polynomial of the broken circuit complex of the same matroid.\\

Section \ref{sec:algebra} (Algebra) deals with the M\"obius algebra of a matroid, which has a $\Z$-basis
given by flats with multiplication given by the join\footnote{In the literature,
one usually sees the multiplication given by meet rather than join.  However, these two products
are isomorphic; indeed, both are isomorphic to the coordinatewise product \cite{Solomon}.
The join product will be more natural for our purposes.}
operation:  $\ep_F\cdot \ep_G = \ep_{F\vee G}$.  We introduce a $q$-deformation of this algebra;
that is, a commutative, associative, unital $\Z[q,q^{-1}]$-algebra with basis given by flats,
such that specializing $q$ to 1 recovers the original M\"obius algebra (Proposition \ref{deformation}).
Using Kazhdan-Lusztig polynomials, we define a new basis whose relationship to the standard basis
is analogous to the relationship between the canonical basis and standard basis for the Hecke algebra,
and we conjecture that the structure coefficients for multiplication in the new basis lie in $\N[q]$ 
(Conjecture \ref{structure coefs}).  We verify this conjecture for Boolean matroids (Proposition \ref{Boolean}),
for uniform matroids of rank at most 3 (Subsection \ref{sec:alg:uni}), and for braid matroids of rank at most 3
(Subsection \ref{sec:alg:braid}).\\\\

\noindent
{\bf Addendum:}  After this paper was published, Ben Young discovered counterexamples to Conjecture \ref{structure coefs}.
See Section \ref{update}.

\vspace{\baselineskip}
\noindent
{\em Acknowledgments:}
The authors would like to thank June Huh,
Joseph Kung, Emmanuel Letellier, Carl Mautner, Hal Schenck, Ben Webster,  Ben Young,
and Thomas Zaslavsky
for their helpful contributions.  The third author is grateful to the University of Oregon
for its hospitality during the completion of this project.

\section{Combinatorics}\label{sec:combinatorics}
In this section we give a combinatorial definition of the Kazhdan-Lusztig polynomial of a matroid,
we compute the first few coefficients, and we study the special cases of uniform matroids and braid matroids.

\subsection{Definition}
Let $M$ be a matroid with no loops on a finite ground set $\cI$.  Let $L(M)\subset 2^\cI$ 
denote the lattice of flats of $M$,
ordered by inclusion, with minimum element $\emptyset$.
Let $\mu$ be the M\"obius function on $L(M)$,
and let $$\chi_M(t) = \sum_{F\in L(M)} \mu(\emptyset,F)\, t^{\rk M - \rk F}$$ 
be the {\bf characteristic polynomial} of $M$.
For any flat $F\in L(M)$, let $\cI^F = \cI\smallsetminus F$ and $\cI_F = F$.  Let $M^F$ be the matroid
on $\cI^F$ consisting of subsets of $\cI^F$ whose union with a basis for $F$ are independent in $M$, 
and let $M_F$ be the matroid on $\cI_F$ consisting of subsets of $\cI_F$ which are independent in $M$.
We call the matroid $M^F$ the {\bf restriction} of $M$ at $F$, and $M_F$ the {\bf localization}
of $M$ at $F$.  (This terminology and notation comes from the corresponding
constructions for arrangements; see Subsection \ref{sec:xa}.)
We have $\rk M^F = \rk M - \rk F$ and $\rk M_F = \rk F$.

\begin{lemma}\label{cancellation}
For any matroid $M$ of positive rank, $\displaystyle\sum_{F\in L(M)}t^{\rk F}\chi_{M_F}(t^{-1})\chi_{M^F}(t) = 0.$
\end{lemma}

\begin{proof}
We have
\begin{eqnarray*}
\sum_{F}t^{\rk F}\chi_{M_F}(t^{-1})\chi_{M^F}(t) &=&
\sum_{F}t^{\rk F}\sum_{E\leq F}\mu(\emptyset,E)\, t^{\rk E-\rk F}\sum_{G\geq F}\mu(F,G)\, t^{\rk M - \rk G}\\\\
&=& \sum_{E\leq F\leq G}\mu(\emptyset,E)\mu(F,G)\, t^{\rk M + \rk E - \rk G}\\\\
&=& t^{\rk M} \sum_{E\leq G}\mu(\emptyset,E)\, t^{\rk E - \rk G} \sum_{F\in[E,G]}\mu(F,G).
\end{eqnarray*}
The internal sum is equal to $\delta(E,G)$ \cite[2.38]{OT}, thus our equation simplifies to
$$\sum_{F}t^{\rk F}\chi_{M_F}(t^{-1})\chi_{M^F}(t) = t^{\rk M} \sum_F \mu(\emptyset,F).$$
This is 0 unless $\rk M = 0$.
\end{proof}

The following is our first main result.

\begin{theorem}\label{KL-exists}
There is a unique way to assign to each matroid $M$ a polynomial $P_M(t)\in\mathbb{Z}[t]$ 
such that the following conditions are satisfied:
\begin{enumerate}
\item If $\rk M = 0$, then $P_M(t) = 1$.
\item If $\rk M > 0$, then $\deg P_M(t) < \tfrac{1}{2}\rk M$.
\item For every $M$, $\displaystyle t^{\rk M} P_M(t^{-1}) = \sum_{F}\chi_{M_F}(t) P_{M^F}(t).$
\end{enumerate}
\end{theorem}

The polynomial $P_M(t)$ will be called the {\bf Kazhdan-Lusztig polynomial} of $M$.  Our proof
of Theorem \ref{KL-exists} closely follows Lusztig's combinatorial proof of the existence of the usual
Kazhdan-Lusztig polynomials \cite[Proposition 2]{Lusztig-Left}, which he attributes to Gabber.

\begin{proof}
Let $M$ be a matroid of positive rank.
We may assume inductively that $P_{M'}(t)$ has been defined for every matroid $M'$
of rank strictly smaller than $\rk M$; in particular, $P_{M^F}(t)$ has been defined for all $\emptyset\neq F\in L(M)$.
Let $$R_M(t) := \sum_{\emptyset\neq F}\chi_{M_F}(t) P_{M^F}(t);$$
then item 3 says exactly that $$t^{\rk M} P_M(t^{-1}) - P_M(t) = R_M(t).$$
It is clear that there can be at most one polynomial $P_M(t)$ of degree strictly less than $\frac{1}{2}\rk M$
satisfying this condition.  The existence of such a polynomial is equivalent to the statement
$t^{\rk M}R_M(t^{-1}) = - R_M(t)$.  

We have
\begin{eqnarray*}
t^{\rk M}R_M(t^{-1}) &=& t^{\rk M}\sum_{\emptyset\neq F}\chi_{M_F}(t^{-1}) P_{M^F}(t^{-1})\\\\
&=& \sum_{\emptyset\neq F} t^{\rk F}\chi_{M_F}(t^{-1}) t^{\rk M^F}P_{M^F}(t^{-1})\\\\
&=& \sum_{\emptyset\neq F\leq G} t^{\rk F}\chi_{M_F}(t^{-1}) \chi_{M^F_G}(t) P_{M^G}(t)\\\\
&=& \sum_{\emptyset\neq G}\left(-\chi_{M_G}(t) P_{M^G}(t) + P_{M^G}(t) \sum_{F\leq G} t^{\rk F}\chi_{M_F}(t^{-1}) \chi_{M^F_G}(t)\right).
\end{eqnarray*}
Since $\rk M_G = \rk G \neq 0$, Lemma \ref{cancellation} says that the internal sum is zero for all 
$G\neq \emptyset$, so our equation simplifies to
$t^{\rk M}R_M(t^{-1}) = - \displaystyle\sum_{\emptyset\neq G} \chi_{M_G}(t) P_{M^G}(t) = -R_M(t)$.
\end{proof}

\begin{conjecture}\label{positivity}
For any matroid $M$, the coefficients of the Kazhdan-Lusztig polynomial $P_M(t)$ are non-negative.
\end{conjecture}

\begin{remark}
In Section \ref{sec:geometry}, we will prove Conjecture \ref{positivity} for representable matroids
by providing a cohomological interpretation of the polynomial $P_M(t)$ (Theorem \ref{comb=geom}).
\end{remark}

Based on our computer computations for uniform matroids and braid matroids (see appendix), 
along with Proposition \ref{lin-non-neg} and Remark \ref{evidence},
we make the following
additional conjecture.  A sequence $e_0,\ldots,e_r$ is called {\bf log concave} if, for all $1<i<r$, 
$e_{i-1}e_{i+1}\leq e_i^2$.
It is said to have {\bf no internal zeros} if the set $\{i\mid e_i\neq 0\}$ is an interval.
Note that a log concave sequence of non-negative integers with no internal zeroes is always unimodal.

\begin{conjecture}\label{log concave}
For any matroid $M$, the coefficients of $P_M(t)$ form a log concave sequence with no internal zeroes.
\end{conjecture}

\begin{remark}
If $M$ is representable, then Huh and Katz proved that the absolute values of the
coefficients of $\chi_M(q)$ form a log concave
sequence with no internal zeroes \cite[6.2]{HuhKatz}, solving a conjure of Read for graphical matroids and 
the representable case of a conjecture of Rota-Heron-Walsh for arbitrary matroids.
\end{remark}

\subsection{Direct sums}
The following proposition says that the Kazhdan-Lusztig polynomial is multiplicative on direct sums.

\begin{proposition}\label{product}
For any matroids $M_1$ and $M_2$, $P_{M_1\oplus M_2}(t) = P_{M_1}(t) P_{M_2}(t)$.
\end{proposition}

\begin{proof}
We proceed by induction.  The statement is clear when $\rk M_1 = 0$ or $\rk M_2 = 0$.  Now assume
that the statement holds for $M_1'$ and $M_2'$ whenever $\rk M'_1\leq \rk M_1$ and 
$\rk M_2'\leq M_2$ with at least one of the two inequalities being strict.

We have $L(M_1\oplus M_2) = L(M_1)\times L(M_2)$.
The localization of $M_1\oplus M_2$ at $(F_1,F_2)$ is isomorphic to $(M_1)_{F_1}\oplus (M_2)_{F_2}$,
and the restriction at $(F_1,F_2)$ is isomorphic to $(M_1)^{F_1}\oplus (M_2)^{F_2}$.
The characteristic polynomial of $(M_1)_{F_1}\oplus (M_2)_{F_2}$ is the product of the two characteristic
polynomials, and our inductive hypothesis tells us that the Kazhdan-Lusztig polynomial of $(M_1)^{F_1}\oplus (M_2)^{F_2}$
is the product of the two Kazhdan-Lusztig polynomials, provided that $F_1\neq \emptyset$ or $F_2\neq \emptyset$.
These two observations, along with the recursive definition of the Kazhdan-Lusztig polynomial, combine to tell us that
$$t^{\rk M_1 + \rk M_2} P_{M_1\oplus M_2}(t^{-1}) -
t^{\rk M_1} P_{M_1}(t^{-1}) \cdot t^{\rk M_2} P_{M_2}(t^{-1})  = - P_{M_1}(t) P_{M_2}(t) + P_{M_1\oplus M_2}(t).$$
The left-hand side is concentrated in degree strictly greater than $\frac 1 2\rk M_1 + \frac 1 2 \rk M_2$,
while the right-hand side is concentrated in degree strictly less than $\frac 1 2\rk M_1 + \frac 1 2 \rk M_2$.
This tells us that both sides must vanish, and the proposition is proved.
\end{proof}

\begin{remark}\label{cheap}
Proposition \ref{product} rules out many potential counterexamples to Conjecture \ref{positivity}.
That is, one cheap way to construct a non-representable matroid is to fix a prime $p$ and let $M = M_1\oplus M_2$,
where $M_1$ is representable only in characteristic $p$ and $M_2$ is representable only in characteristic $\neq p$.
Proposition \ref{product} will tells that the Kazhdan-Lusztig polynomial of $M_1\oplus M_2$
is equal to the product of the Kazhdan-Lusztig polynomials of $M_1$ and $M_2$, each of which has 
non-negative coefficients because $M_1$ and $M_2$ are both representable.
\end{remark}

\begin{remark}
Proposition \ref{product} is also consistent with Conjecture \ref{log concave}, since
the convolution of two non-negative log concave sequences with no internal zeroes is again log concave
with no internal zeroes \cite[Theorem 1]{KookLC}.  Note that the corresponding statement would be false without
the no internal zeroes hypothesis.\footnote{We thank June Huh for pointing out this fact.}
\end{remark}

\begin{corollary}\label{boolean}
If $M$ is the Boolean matroid on any finite set, then $P_M(t) = 1$.
\end{corollary}

\begin{proof}
The Boolean matroid on a set of cardinality $n$ is isomorphic to the direct sum of $n$ copies of
the unique rank 1 matroid on a set of cardinality $1$.
\end{proof}

\subsection{The first few coefficients}
In this subsection we interpret the first few coefficients of $P_M(t)$ in terms of the 
doubly indexed Whitney numbers of $M$, introduced by Green and Zaslavsky \cite{GreenZaslavsky}.

\begin{proposition}\label{constant term}
The constant term of $P_M(t)$ is equal to 1.
\end{proposition}

\begin{proof}
We proceed by induction on the rank of $M$.
If $\rk M = 0$, then $P_M(t) = 1$ by definition.  If $\rk M > 0$, we consider the recursion
$$t^{\rk M} P_M(t^{-1}) = \sum_{F}\chi_{M_F}(t) P_{M^F}(t).$$
Since $\deg P_M(t) < \rk M$, the left-hand side has no constant term, therefore we have
$$0 = \sum_{F}\chi_{M_F}(0) P_{M^F}(0).$$
By our inductive hypothesis, we may assume that $P_{M^F}(0) = 1$ for all nonempty flats $F$,
and we therefore need to show that 
\begin{equation*}\label{want}
0 = \sum_{F}\chi_{M_F}(0).
\end{equation*}
This follows from the fact that $\chi_{M_F}(0) = \mu(\emptyset,F)$ and $\rk M> 0$.
\end{proof}

For all natural numbers $i$ and $j$, let 
$$w_{i,j}\; := \sum_{\rk E = i, \rk F = j} \mu(E,F)
\and
W_{i,j}\; := \sum_{\rk E = i, \rk F = j} \zeta(E,F),$$
where $\zeta(E,F) = 1$ if $E\leq F$ and 0 otherwise.
These are called {\bf doubly indexed Whitney numbers} of the first and second kind, respectively.
In the various propositions that follow, we let $d=\rk M$.

\begin{proposition}\label{linear term}
The coefficient of $t$ in $P_M(t)$ is equal to $W_{0,d-1} - W_{0,1}.$
\end{proposition}

\begin{proof}
We consider the defining recursion 
$$t^{\rk M} P_M(t^{-1}) = \sum_{F}\chi_{M_F}(t) P_{M^F}(t)$$
and compute the coefficient of $t^{\rk M -1}$ on the right-hand side.
The flat $F = \cI$ contributes $-W_{0,1}$, and each of the $W_{0,d-1}$
flats of rank $d-1$ contributes 1.
\end{proof}

\begin{remark}
If $M$ is the matroid associated to a hyperplane arrangement, Proposition \ref{linear term}
says that the coefficient of $t$ in $P_M(t)$ is equal to the number of lines in the lattice of flats
minus the number of hyperplanes.
\end{remark}




\begin{proposition}\label{lin-non-neg}
The coefficient of $t$ in $P_M(t)$ is always non-negative, and the following are equivalent:
\begin{enumerate}
\item[(i)] $P_M(t) = 1$
\item[(ii)] the coefficient of $t$ is zero
\item[(iii)] the lattice $L(M)$ is modular.
\end{enumerate}
\end{proposition}

\begin{proof}
Non-negativity of the linear term follows from Proposition \ref{linear term} 
along with the hyperplane theorem \cite[8.5.1 \& \S8.5]{Aigner}.  The hyperplane theorem also states that the linear
term is zero if and only if $L(M)$ is modular.
The first item obviously implies the second, so it remains only to show that $P_M(t) = 1$
whenever $L(M)$ is modular.

We proceed by induction on $d=\rk M$.  The base case is trivial.  Assume the statement holds
for all matroids of rank smaller than $d$, and that $L(M)$ is modular.  In particular, for any flat $F$, $L(M^F)$
is also modular, so we may assume that $P_{M^F}(t) = 1$ for all $F\neq\emptyset$.  Thus the defining
recursion says that
$$t^dP_M(t^{-1}) - P_M(t) = \sum_{F\neq\emptyset} \chi_{M_F}(t),$$
and we need only show that the right-hand side is equal to $t^d-1$.  Equivalently, we need to show that
$\sum_{F} \chi_{M_F}(t)$ is equal to $t^d$.

Since $L(M)$ is modular, there exists another matroid $M'$ such that $L(M')$ is dual to $L(M)$;
that is, there exists an order-reversing and rank-reversing bijection between $L(M)$ and $L(M')$.
(This is simply the statement that the dual of $L(M)$ is again a geometric lattice, which follows
from modularity.)  This implies that
$$\sum_{F} \chi_{M_F}(t) = \sum_{F,G} \mu_M(F,G) t^{d - \rk_M G}
= \sum_{F,G} \mu_{M'}(G,F) t^{\rk_{M'}G}.$$
By M\"obius inversion, this sum vanishes in all degrees less than $d$, and the coefficient of $t^d$
is equal to $\mu_{M'}(\cI,\cI) = 1$.
\end{proof}

\begin{remark}\label{evidence}
Note that the implication of (i) by (ii) in Proposition \ref{lin-non-neg} provides evidence
for the lack of internal zeroes in the sequence of coefficients of $P_M(t)$ (Conjecture \ref{log concave}).
\end{remark}

\begin{proposition}\label{quadratic term}
The coefficient of $t^2$ in $P_M(t)$ is equal to $$w_{0,2} - W_{1,d-1} + W_{0,d-2} - W_{d-3,d-2} + W_{d-3,d-1}.$$
\end{proposition}

\begin{proof}
We again consider the defining recursion, and this time we
compute the coefficient of $t^{\rk M -2}$ on the right-hand side.
The flat $F=\cI$ contributes $w_{0,2}$,
each flat $F$ of rank $d-1$ contributes $-W_{0,1}(M_F)$, and $-\sum_{\rk F = d-1}W_{0,1}(M_F) = -W_{1,d-1}$.
Each of the $W_{0,d-2}$ flats of rank $d-2$ contributes 1.
Each flat $F$ of rank $d-3$ contributes the linear term of $P_{M^F}(t)$, which is equal to
$W_{0,2}(M^F) - W_{0,1}(M^F)$ by Proposition \ref{linear term}.  Summing over all such flats,
we obtain the final two terms $W_{d-3,d-1} - W_{d-3,d-2}$.
\end{proof}

\begin{remark}\label{half of life}
We have $w_{0,2} = W_{1,2} - W_{0,2},$
$$W_{1,2} - W_{1,d-1} = \sum_{\rk F = 1}\left(W_{0,1}(M^F) - W_{0,d-2}(M^F)\right),$$
and
$$W_{d-3,d-1} - W_{d-3,d-2} = \sum_{\rk G = d-3}\left(W_{0,2}(M^G) - W_{0,1}(M^G)\right),$$
thus the coefficient of $t^2$ is equal to
$$\sum_{\rk F = 1}\Big(W_{0,1}(M^F) - W_{0,d-2}(M^F)\Big)
\;\;+ \sum_{\rk G = d-3}\Big(W_{0,2}(M^G) - W_{0,1}(M^G)\Big)
\;\;+\;\; \Big(W_{0,d-2} - W_{0,2}\Big).$$
The hyperplane theorem says that each of the summands in the first sum is non-positive
and each of the summands in the second sum is non-negative.
The statement that $W_{0,d-2} - W_{0,2}$ is non-negative as long as $d\geq 4$ is a long standing
conjecture in matroid theory, called the ``top-heavy conjecture" 
\cite{DW75}, \cite[2.5.2]{Kung}.
(Note that if $d\leq 4$, then the coefficient of $t^2$ in $P_M(t)$ is automatically zero.)
Thus a comparison (in either direction) between the absolute values of the two sums would
yield a logical implication (in the corresponding direction) between 
Conjecture \ref{positivity} (for quadratic terms) and the top-heavy conjecture.
\end{remark}

The next proposition, whose proof we omit, indicates the difficulty with finding a closed formula for these coefficients. On the other hand, \cite[5.5]{Wak} presents a formula for all coefficients, albeit recursively defined, in the same vein as Proposition \ref{quadratic term}.



\begin{proposition}\label{cubic term}
The coefficient of $t^3$ in $P_M(t)$ is equal to 
\begin{eqnarray*}
&& w_{0,3} -W_{d-4,d-3}+ W_{d-4,d-1}-W_{1,d-2} +W_{0,d-3}\\\\
+&&
\sum\limits_{\rk F = d-1}w_{0,2}(M_F) - \sum\limits_{\rk F = d-3}W_{0,1}(M_F)\Big[W_{0,2}(M^F)-W_{0,1}(M^F)\Big]
\\\\
+&&\sum\limits_{\rk F = d-5}\Big[ w_{0,2}(M^F)-W_{1,4}(M^F)+W_{0,3}(M^F)+W_{2,4}(M^F)-W_{2,3}(M^F)\Big].
\end{eqnarray*}
\end{proposition}

\subsection{Uniform matroids}
Given non-negative integers $d$ and $m$, let $M_{m,d}$ be the uniform matroid of rank $d$ on a set of cardinality $m+d$,
and write $$P_{M_{m,d}}(t) = P_{m,d}(t) =  \sum_i c^i_{m,d}\; t^i.$$
The values of $P_{m,d}(t)$ for small $m$ and $d$ appear in the appendix.

For any flat $F$ of rank strictly less than $d$, the localization $(M_{m,d})_F$ is a Boolean matroid,
and the restriction $M_{m,d}^F$ is isomorphic to $M_{m,d-\rk F}$, thus our recursive definition will give us a recursive
relation among the coefficients $c^i_{m,d}$ for a single fixed $m$.  Specifically, we have the following result
(the factor before $c^j_{m,k}$ is a trinomial coefficient).

\begin{proposition}\label{uniform coefficient recursion}
For any $m$, $d$, and $i$, we have
$$c_{m,d}^i = (-1)^i\binom{m+d}{i} + 
\sum_{j=0}^{i-1}\sum_{k=2j+1}^{i+j}(-1)^{i+j+k}\binom{m+d}{m+k,i+j-k,d-i-j} c^j_{m,k}.$$
\end{proposition}

We can use Proposition \ref{uniform coefficient recursion} obtain explicit formulas for the first few coefficients.  
In general, the formula
for $c_{m,d}^i$ will be a signed sum of $(i+1)$-nomial coefficients, each with $m+d$ on top.
We omit the proof of Corollary \ref{uniform coefficients} because it is a straightforward application
of the proposition.

\begin{corollary}\label{uniform coefficients}
We have
\begin{eqnarray*}
c_{m,d}^0 &=& 1\\\\
c_{m,d}^1 &=& \tbinom{m+d}{m+1} -\tbinom{m+d}{1} \\\\
c_{m,d}^2 &=&
\tbinom{m+d}{m+1,d-3,2}
- \tbinom{m+d}{m+1,d-2,1}
+ \tbinom{m+d}{m+2,d-2,0}
- \tbinom{m+d}{m+2,d-3,1}
+ \tbinom{m+d}{2}\\\\
c_{m,d}^3 &=& 
\;\;\;\tbinom{m+d}{m+1,d-3,2,0}
- \tbinom{m+d}{m+1,d-4,2,1}
+ \tbinom{m+d}{m+1,d-4,3,0}
- \tbinom{m+d}{m+1,d-5,3,1}
+ \tbinom{m+d}{m+1,d-5,2,2}\\
&& - \tbinom{m+d}{m+2,d-3,1,0}
+ \tbinom{m+d}{m+2,d-4,1,1}
- \tbinom{m+d}{m+2,d-5,2,1}
+ \tbinom{m+d}{m+2,d-5,3,0}
+ \tbinom{m+d}{m+3,d-3,0,0}\\
&& - \tbinom{m+d}{m+3,d-4,1,0}
+ \tbinom{m+d}{m+3,d-5,2,0}
-\tbinom{m+d}{3}.
\end{eqnarray*}
\end{corollary}

We can also express our recursion in terms of a generating function identity.
Let $$\Phi_m(t,u) = \sum_{d=1}^{\infty} P_{m,d}(t) u^d.$$

\begin{proposition}\label{uniform gf}
We have $$\Phi_m(t^{-1}, tu) = \frac{tu-u}{(1-tu+u)(1+u)^{m}}
+ \frac{1}{(1-tu+u)^{m+1}} \;\Phi_m\!\left(t,\frac{u}{1-tu+u}\right).$$
\end{proposition}

\begin{proof}
Our defining recursion tells us that
\begin{eqnarray*}\Phi_m(t^{-1}, tu) &=& \sum_{d=1}^{\infty} P_{m,d}(t^{-1}) t^d u^d\\\\
&=& \sum_{d=1}^{\infty} \left[\sum_{i=0}^d (-1)^i\binom{m+d}{i} (t^{d-i}-1) + \sum_{k=1}^{d} 
\binom{m+d}{d-k} (t-1)^{d-k} P_{m,k}(t) \right]u^d.
\end{eqnarray*}
If we introduce new dummy indices $e = d-i$ and $f = d-k$, we may rewrite this equation as
$$\Phi_m(t^{-1}, tu) = \sum_{e=0}^{\infty}(t^{e}-1)u^e \sum_{i=0}^{\infty}\binom{m+e+i}{i}(-u)^i
+ \sum_{k=1}^\infty P_{m,k}(t)u^k\sum_{f=0}^\infty\binom{m+k+f}{f}(ut-u)^f.$$
Next, we recall that $$\sum_{\ell=0}^\infty\binom{r+\ell}{\ell}x^\ell = \frac{1}{(1-x)^{r+1}}.$$
We will use this formula with $r=m+e$ and $x = -u$, and then again with $r=m+k$ and $x = tu-u$, to get
\begin{eqnarray*}\Phi_m(t^{-1}, tu) &=& \sum_{e=0}^{\infty} \frac{(t^{e}-1)u^e}{(1+u)^{m+e+1}}
+ \sum_{k=1}^\infty \frac{P_{m,k}(t) u^k}{(1-tu+u)^{m+k+1}}\\\\
&=& \frac{1}{(1+u)^{m+1}}\sum_{e=0}^{\infty} \frac{(t^{e}-1)u^e}{(1+u)^{e}}
+ \frac{1}{(1-tu+u)^{m+1}}\sum_{k=1}^\infty P_{m,k}(t) \left(\frac{u}{1-tu+u}\right)^k\\\\
&=& \frac{tu-u}{(1-tu+u)(1+u)^{m}}
+ \frac{1}{(1-tu+u)^{m+1}} \;\Phi_m\!\left(t,\frac{u}{1-tu+u}\right).
\end{eqnarray*}
This completes the proof.
\end{proof}

\begin{remark}
A general formula for $c_{m,d}^i$ can be obtained from \cite[3.1]{GPY} and the ensuing remarks.
\end{remark}

\subsection{Braid matroids}
Let $M_n$ be the braid matroid of rank $n-1$; this is the matroid associated with the complete graph
on $n$ vertices, or with the braid arrangement (Example \ref{braid-OT}).
The lattice $L(M_n)$ is isomorphic to the lattice of set-theoretic partitions of the set $[n]$.  Let $P_n(t) = P_{M_n}(t)$.
Values of $P_n(t)$ for $n\leq 20$ appear in the appendix.

For any partition $\la$ of the number $n$, let 
$$m(\la) := \frac{n!}{\prod_{i=1}^{\ell(\la)}\la_i!\cdot\prod_{j=1}^{\la_1}(\la^t_j - \la^t_{j+1})!}$$
be the number of flats of type $\la$, where $\la^t$ denotes the transpose
partition and $\ell(\la)$ is the number of parts of $\la$.
For such a flat $F$, the localization $(M_n)_F$ is isomorphic to $M_{\la_1}\oplus\cdots\oplus M_{\la_{\ell(\la)}}$,
and has characteristic polynomial
$$\chi(t) = \prod_{i=1}^{\ell(\la)} (t-1)\cdots(t-\la_i+1) = \prod_{j=1}^{\la_1-1} (t-j)^{\la^t_{j+1}}.$$
The restriction $M_n^F$ is isomorphic (after simplification) to $M_{\ell(\la)}$.

The Whitney numbers of the $M_n$ can be interpreted in terms of Stirling numbers
of the first and second kind, respectively.
By definition,
$$s(n,k) := w_{0,n-k}\and S(n,k) := W_{0,n-k}.$$

\begin{lemma}\label{W-S}
For all $i\leq j$, $W_{i,j} = S(n,n-i)S(n-i,n-j)$.
\end{lemma}

\begin{proof}
A flat of rank $i$ corresponds to a partition of $[n]$ into $n-i$ blocks, and there are $W_{0,i} = S(n,n-i)$
such flats.  For each such flat, a flat of rank $j$ lying above it corresponds to a partition of the set of blocks
into $n-j$ blocks, and there are $S(n-i,n-j)$ such flats.
\end{proof}

\begin{corollary}
The coefficient of $t$ in $P_n(t)$ is equal to $S(n,2)-S(n,n-1)$,
and the coefficient of $t^2$ is equal to $s(n,n-2)-S(n,n-1)S(n-1,2)+S(n,3)+S(n,4)$.
\end{corollary}

\begin{proof}
This follows from Proposition \ref{linear term}, Proposition \ref{quadratic term}, and Lemma \ref{W-S},
along with the observation that $w_{0,2} = s(n,n-2)$.
\end{proof}

\begin{lemma}\label{Whitney product}
For any matroids $M$ and $M'$,
$$W_{i,j}(M\oplus M')=\sum_{k,\ell}W_{k,\ell}(M)W_{i-k,j-\ell}(M').$$
\end{lemma}

\begin{proof}
This follows from the fact that $L(M\oplus M') = L(M)\times L(M')$ as ranked posets.
\end{proof} 

The following proposition, which may be derived from Proposition \ref{cubic term}, Lemma \ref{W-S},
and Lemma \ref{Whitney product},
expresses the cubic term of $P_n(t)$ in terms of Stirling numbers and binomial coefficients.
More generally, since any restriction of a braid matroid is another braid matroid and any localization
of a braid matroid is a direct sum of braid matroids, it would be possible to express every coefficient of
$P_n(t)$ in terms of Stirling numbers and binomial coefficients.

\begin{proposition}
The coefficient of $t^3$ in $P_n(t)$ is equal to 
\begin{eqnarray*}
s(n,n-3)&+&\sum_{\la\vdash n, \ell(\la)=4}
m(\la)\Big[ S(\la_1,\la_1-1)S(\la_1-1,\la_1-2)+S(\la_2,\la_2-1)S(\la_1,\la_1-1)\\
&&\qquad\qquad\qquad\;\; +S(\la_2,\la_2-1)S(\la_2-1,\la_2-2)-S(\la_1,\la_1-2)-S(\la_2,\la_2-2)\Big]\\\\
&-&S(n,n-1)S(n-1,3)+S(n,4)\\ \\
&+&\sum\limits_{\la\vdash n, \ell(\la)=4}m(\la)\left( \tbinom{\la_1}{2}+\tbinom{\la_2}{2}+\tbinom{\la_3}{2}+\tbinom{\la_4}{2}\right)\\\\
&+&5S(n,5)+15S(n,6).
\end{eqnarray*} 
\end{proposition}

Finally, we express the recursion for the polynomials $P_n(t)$ as a generating function identity, 
just as we did for uniform matroids.
Let $$\Psi(t,u) = \sum_{n=1}^\infty P_n(t) u^{n-1}.$$
For any partition $\nu$ (of any number), let $\tilde\nu$ be the partition of $|\nu|+\ell(\nu)$
obtained by adding 1 to each of the parts of $\nu$.

\begin{proposition}\label{braid gf}
We have
$$\Psi(t^{-1}, tu) = \sum_{\nu} m(\tilde\nu)u^{|\tilde\nu|-1}\prod_{j=1}^{\nu_1} (t-j)^{\nu^t_j}\cdot
\frac{\partial_u^{|\tilde\nu|}}{|\tilde\nu|!}
\left(u^{|\nu|+1}\Psi(t,u)\right),$$
where the sum is over all partitions $\nu$ of any size.
\end{proposition}

\begin{proof}
Our defining recursion tells us that
\begin{eqnarray*}
\Psi(t^{-1}, tu) = \sum_{n=1}^\infty P_n(t^{-1}) t^{n-1} u^{n-1}
&=& \sum_{n=1}^\infty\left[\sum_{\la\vdash n} m(\la) P_{\ell(\la)}(t)\prod_{j=1}^{\la_1-1} (t-j)^{\la^t_{j+1}}\right] u^{n-1}\\\\
&=& \sum_{\la} m(\la) P_{\ell(\la)}(t) u^{|\la|-1}\prod_{j=1}^{\la_1-1} (t-j)^{\la^t_{j+1}}.
\end{eqnarray*}
(We adopt the convention that $P_0(t)=0$
so that the empty partition contributes nothing to the sum.)
For any partition $\nu$, let 
$\tilde\nu_k$ be the partition obtained by adding $k$ new parts of size 1 to $\tilde\nu$.
We will replace the sum over $\la$ with a sum over $\nu$ and $k$, with $\la = \tilde\nu_k$.
Note that we have $$m(\la) = \binom{|\tilde\nu_k|}{|\tilde\nu|}m(\tilde\nu),
\qquad \ell(\la) = \ell(\nu)+k,
\qquad \la_1 - 1 = \nu_1,
\and \la^t_{j+1} = \nu^t_j,$$
thus we can rewrite our equation as
\begin{eqnarray*}
\Psi(t^{-1}, tu) &=& \sum_{\nu}\sum_{k=0}^\infty \binom{|\tilde\nu_k|}{|\tilde\nu|}m(\tilde\nu)
P_{\ell(\nu)+k}(t)u^{|\tilde\nu_k|-1}
\prod_{j=1}^{\nu_1} (t-j)^{\nu^t_j}\\\\
&=& \sum_{\nu} m(\tilde\nu)u^{|\nu|}\prod_{j=1}^{\nu_1} (t-j)^{\nu^t_j}\cdot\sum_{k=0}^\infty
\binom{|\tilde\nu_k|}{|\tilde\nu|}
P_{\ell(\nu)+k}(t)u^{\ell(\nu)+k-1}.
\end{eqnarray*}
Next, we observe that 
$$\binom{|\tilde\nu_k|}{|\tilde\nu|} u^{\ell(\nu)+k-1} 
= u^{\ell(\nu)-1}\frac{\partial_u^{|\tilde\nu|}}{|\tilde\nu|!}u^{|\tilde\nu_k|},$$
so
\begin{eqnarray*}
\Psi(t^{-1}, tu) &=&
\sum_{\nu} m(\tilde\nu)u^{|\nu|}\prod_{j=1}^{\nu_1} (t-j)^{\nu^t_j}\cdot\sum_{k=0}^\infty
u^{\ell(\nu)-1}\frac{\partial_u^{|\tilde\nu|}}{|\tilde\nu|!}u^{|\tilde\nu_k|}
P_{\ell(\nu)+k}(t)\\\\
&=& \sum_{\nu} m(\tilde\nu)u^{|\tilde\nu|-1}\prod_{j=1}^{\nu_1} (t-j)^{\nu^t_j}\cdot
\frac{\partial_u^{|\tilde\nu|}}{|\tilde\nu|!}
\left(u^{|\nu|+1}\sum_{k=0}^\infty u^{\ell(\nu)+k-1}
P_{\ell(\nu)+k}(t)\right)\\\\
&=& \sum_{\nu} m(\tilde\nu)u^{|\tilde\nu|-1}\prod_{j=1}^{\nu_1} (t-j)^{\nu^t_j}\cdot
\frac{\partial_u^{|\tilde\nu|}}{|\tilde\nu|!}
\left(u^{|\nu|+1}\Psi(t,u)\right).
\end{eqnarray*}
This completes the proof.
\end{proof}

\excise{
\subsection{Some non-representable matroids}
The V\'amos matroid $M$ is the smallest non-representable matroid; the ground set has cardinality 8,
the matroid has rank 4, and all circuits have cardinality 4 or 5.  
There are 5 circuits of cardinality 4, each of which is a flat of rank 3.
All other flats of rank 3 have cardinality 3, and the number of them is 
$\binom{8}{3} - 5 \times \binom{4}{3} = 36$.  There are 8 flats of rank 1, so the coefficient of $t$
in the Kazhdan-Lusztig polynomial of $M$ is $36 - 8 = 28$.  Thus $P_M(t) = 1 + 28t$.
\nicktodo{Let's see some examples where we have large enough coefficients that we don't {\em a priori}
know that they will be non-negative.}
}

\section{Geometry}\label{sec:geometry}
In this section we give a cohomological interpretation of the polynomial $P_M(t)$ whenever the matroid $M$ is representable; this interpretation is analogous to the interpretation of Kazhdan-Lusztig polynomials
associated to Weyl groups as local intersection cohomology groups of Schubert varieties \cite{KL80}. 
In particular, we prove Conjecture \ref{positivity} for representable matroids.

\subsection{The reciprocal plane}\label{sec:xa}
Let $k$ be a field.  An {\bf arrangement} $\cA$ over $k$ is a triple $(\cI, V, a)$,
where $\cI$ is a finite set, $V$ is a finite dimensional vector space over $k$, and $a$ is a map
from $\cI$ to $V^*\smallsetminus\{0\}$ such that the image of $a$ spans $V^*$.
Let $$U_\cA := \{v\in V\mid \langle a(i), v\rangle\neq 0\;\text{for all $i\in\cI$}\};$$
this variety is called the {\bf complement} of $\cA$.
We have a natural inclusion of $U_\cA$ into $(k^\times)^\cI$ whose $i^\text{th}$ coordinate is given by $a(i)$.
Consider the involution of $(k^\times)^\cI$ obtained by inverting every coordinate, and let $U_\cA^{-1}$ be the image
of $U_\cA$ under this involution.  The {\bf reciprocal plane} $X_\cA$ is defined to be the closure of $U_\cA^{-1}$
inside of $k^\cI$.  Its coordinate ring $k[X_\cA]$ is isomorphic to the subalgebra of $k(V)$ generated by 
$\{a(i)^{-1}\mid i\in\cI\}$; this ring is called the {\bf Orlik-Terao algebra}.

Consider the polynomial ring $k[u]_\cI$ with generators $\{u_i\mid i\in \cI\}$.  For all $S\subset \cI$,
let $$u_S := \prod_{i\in S}u_i.$$
Consider the surjective map $\rho:k[u]_\cI\to k[X_\cA]$ taking $u_i$ to $a(i)^{-1}$.
Suppose that $c\in k^\cI$ has the property that $\sum c_i a(i) = 0$;
we call such a vector a {\bf dependency} for $\cA$. 
Let $S_c := \{i\in\cI\mid c_i\neq 0\}$ be the support of $c$, and for all $i\in S_c$, let $S_c^i = S_c\smallsetminus\{i\}$.  
Then we obtain an element
$$f_c(u) := \sum_{i\in S_c}c_i u_{S_c^i} \in \ker(\rho).$$
Indeed, if we take the polynomials $f_c$ associated to vectors $c$ of minimal support,
we obtain a universal Gr\"obner basis for the kernel of $\rho$ \cite[Theorem 4]{PS}.
Note that the kernel of $\rho$ is a homogeneous ideal, thus inducing a grading on $k[X_\cA]$.

Let $M_\cA$ be the matroid with ground set $\cI$ consisting of subsets of $\cI$ on which $a$ is injective
with linearly independent image.  We say that $\cA$ {\bf represents} $M_\cA$ over $k$.
Given a flat $F$, let $\cI^F = \cI\smallsetminus F$ and $\cI_F = F$.  Let
$$V^F := \operatorname{Span}\{a(i)\mid i\in F\}^\perp \subset V\and V_F := V/V^F,$$
and consider the natural maps 
$$a^F:\cI^F\to (V^F)^*\and a_F:\cI_F\to V_F^*.$$
We define the {\bf restriction} $\cA^F := (\cI^F, V^F, a^F)$ and the {\bf localization} $\cA_F := (\cI_F, V_F, a_F)$.
Then we have
$$M_{\cA^F} = M^F\and M_{\cA_F} = M_F.$$

For any subset $F\subset \cA$, let $X_{\cA,F}$ be the subvariety of $X_\cA\subset k^\cI$ 
consisting of points whose $i^\text{th}$ coordinate vanishes if and only if $i\notin F$.
The following result is proved in \cite[Proposition 5]{PS}.

\begin{proposition}\label{strat}
The subvariety $X_{\cA,F}\subset X_\cA$ is nonempty if and only if $F$ is a flat, in which case it is isomorphic to $U_{\cA_F}$, and its closure is isomorphic to $X_{\cA_F}$.
\end{proposition}

\begin{example}\label{braid-OT}
Let $V = k^n/k_\Delta$, and let $\cA$ be the {\bf braid arrangement} consisting of all linear functionals of the form $x_i-x_j$, where $i<j$.
Flats of $\cA$ correspond to set-theoretic partitions of $[n]$; the restrictions $\cA^F$ are smaller braid arrangements (with multiplicities),
while the localizations $\cA_F$ are products of smaller braid arrangements.

The complement $U_\cA$ is the set of distinct ordered $n$-tuples of points in $k$ up to simultaneous translation.  In the closure of $U_\cA$, 
distances between points may go to zero (that is, the points are allowed to collide).  When they do, you see the complement of a restriction of $\cA$.
In the closure of $U_\cA^{-1}$, distances between points may go to infinity,
which means that our set of $n$ points may split into a disjoint union of smaller sets, each of which lives in a ``far away" copy of $k$.
When they do, you see the complement of a localization of $\cA$.
\end{example}

\subsection{Local geometry of the reciprocal plane}\label{sec:local}
For any flat $F$ of $\cA$,  let $W_{\!\cA,F}\subset X_\cA$ be the open subvariety 
defined by the nonvanishing of $u_i$ for all $i\in F$.  Equivalently, $W_{\!\cA,F}$ is the preimage of
$X_{\cA,F}$ along the canonical projection $\pi:X_\cA\to X_{\cA_F}$ given by setting the coordinates
in $\IF$ to zero.
The following theorem will be the main ingredient in our proof of Theorem \ref{comb=geom},
which gives a cohomological interpretation of the Kazhdan-Lusztig polynomial of a representable matroid.
It says roughly that the reciprocal plane $X_{\!\cA^F}$ associated to the restriction $\cA^F$ is an ``\'etale slice"
to the stratum $X_{\cA,F}\subset X_\cA$.

\begin{theorem}\label{etale}
Let $F$ be a flat of $\cA$ and let $x\in X_{\cA,F}\subset X_\cA$.  Then there exists 
an open subscheme $\becircled W_{\!\cA,F}\subset W_{\!\cA,F}$ containing $x$
and a map $\Phi:\becircled W_{\!\cA,F}\to X_{\!\cA^F}\times X_{\cA,F}$
such that $\Phi(x) = (0,x)$ and $\Phi$ is \'etale at $x$.
\end{theorem}

\begin{proof}
Consider the natural projection from $V$ to $V_F$, and choose a splitting $\sigma:V_F\to V$
of this projection.  Let $\iota:X_{\cA,F}\to U_{\cA_F}$ be the isomorphism mentioned in Proposition \ref{strat}.
Concretely, $X_{\cA,F}$ and $U_{\cA_F}$ are both subschemes of $(k^\times)^F$, and $\iota$ is given
by inverting all of the coordinates.
For all $j\in \IF$, let $$b_j := \pi^*\iota^*\sigma^* u_j \in k[W_{\!\cA,F}].$$
Here we regard $u_j\in k[u]_\cI$ as a function on $V\subset k^\cI$, so that $\sigma^*u_j$ is a function
on $V_F$, and therefore on $U_{\cA_F}\subset V_F$.  Then $\iota^*\sigma^* u_j$ is a function on $X_{\cA,F}$,
and $b_j$ is its pullback to $W_{\!\cA,F}$.
By construction of $b_j$,
we have 
\begin{equation}\label{bj}
\sum_{i\in F}c_i u_i^{-1} + \sum_{j\in\IF}c_j b_j = 0 \in k[W_{\!\cA,F}]
\end{equation}
for any dependency $c$ of $\cA$.

Let $\becircled W_{\!\cA,F}$ be the open subscheme of $W_{\!\cA,F}$ defined
by the nonvanishing of $1-b_ju_j$ for all $j\in\IF$. 
Since $u_j$ vanishes at $x$ for all $j\in\IF$, we have
$x\in \becircled W_{\!\cA,F}$.
Recall that $$k[X_\cA] \cong k[u]_\cI\Big{/}\left\langle f_c(u) \mid \text{$c\in k^\cI$ a dependency}\right\rangle.$$
For any dependency $c$, let 
$\bc$ be the projection of $c$ onto $k^\IF$.
Then $\bc$ is a dependency for $\cA^F$, and all dependencies for $\cA^F$ arise in this way, thus
\begin{equation*}\label{ringR}
k[X_{\!\cA^F}] \cong k[u]_{\IF}\Big{/}\left\langle f_{\bc}(u) \mid \text{$c\in k^\cI$ a dependency}\right\rangle.
\end{equation*}
We define the map $$\varphi:k[X_{\!\cA^F}]\to k[\becircled W_{\!\cA,F}]$$
by putting $$\varphi(u_j) = \frac{u_j}{1-b_ju_j}$$
for all $j\in \IF$.  To show that this is well-defined, we must show that $f_{\bc}(u)$ maps to zero.
Indeed, we have
\begin{eqnarray*}
\varphi(f_{\bc}(u)) &=& \sum_{j\in S_{\bc}} c_j \prod_{k\in S_{\bc}^j} \frac{u_k}{1 - b_ku_k}\\\\
&=& \frac{\sum_{j\in S_{\bc}} c_j u_{S_{\bc}^j} (1-b_ju_j)}{\prod_{k\in S_{\bc}} (1- b_ku_k)}\\\\
&=& \frac{f_{\bc}(u) - \sum_{j\in\IF} c_jb_j\, u_{S_{\bc}}}{\prod_{k\in S_{\bc}} (1- b_ku_k)}\\\\
&=& \frac{f_{\bc}(u) + \sum_{i\in F} c_iu_i^{-1}\, u_{S_{\bc}}}{\prod_{k\in S_{\bc}} (1- b_ku_k)}\\\\
&=& \frac{f_c(u)}{u_{S_c\cap F}\prod_{k\in S_{\bc}} (1- b_ku_k)}.
\end{eqnarray*}
Since $f_c(u)$ vanishes on $X_\cA$, it vanishes on $\becircled W_{\!\cA,F}\subset X_\cA$, as well.

Now consider the map $\Phi:\becircled W_{\!\cA,F}\to X_{\!\cA^F}\times X_{\cA,F}$ induced by $\varphi$ on the first factor
and given by $\pi$ on the second factor.  Since $\pi(x) = x$ and $u_j$ vanishes on $x$ for all $j\in\IF$, we have
$\Phi(x) = (0,x)$.  The statement that $\Phi$ is \'etale at $x$ is equivalent to the statement
that $\Phi$ induces an isomorphism on tangent cones.  Indeed, the tangent cone of $X_{\!\cA^F}\times X_{\cA,F}$
at $(0,x)$ is isomorphic to $X_{\!\cA^F}\times V_F$, and the same is true of the tangent cone
of $X_\cA$ at $x$ \cite[Theorem 24]{SSV}.  The fact that $\Phi$ induces an isomorphism
follows from the fact that, for all $i\in F$, $\pi^*(u_i) = u_i$, and for all $j\in \IF$, $\varphi(u_j) = u_j + O(u_j^2)$.
\end{proof}

\subsection{Intersection cohomology}\label{sec:ih}
The purpose of this subsection is to introduce and prove Theorem \ref{chastity}.  This is a slight reformulation
of \cite[4.1]{PW07}, which was in turn based on the work in \cite[\S 4]{KL80}.  See also \cite[3.3.3]{Let13}
for a similar result, formulated in Hodge theoretic terms, with a slightly different set of hypotheses.

Let $X$ be a variety over a finite field $\Fq$.
Fix a prime number $\ell$ not dividing $q$, and consider the $\ell$-adic
\'etale intersection cohomology group $\IHX := \H^{*-\dim X}(X; \IC_X)$.  Let $\Fr$ be the Frobenius automorphism of $X$,
and let $\Fr^i$ be the induced automorphism of $\IHiX$.
We say that $X$ is {\bf pure} if the eigenvalues of $\Fr^i$ all have absolute value equal to $q^{i/2}$.
We say that $X$ is {\bf chaste} if $\IHiX = 0$ for all odd $i$ and $\Fr^{2i}$ acts 
by multiplication by $q^i \in \Z \subset \Ql$ on $\IH^{2i}(X; \Ql)$.
If $X$ is chaste, then we define
$$P_X(t) := \sum_{i\geq 0} \dim \IH^{2i}(X; \Ql)\; t^i,$$
so that $P_X(q^s) = \operatorname{tr}\!\big((\Fr^*)^s\big)$.

Given a point $x\in X$, we will also be interested in the local intersection cohomology groups
$\IHxX := \H^{*-\dim X}(\IC_{X,x})$.  We say that $X$ is {\bf pointwise pure}
or {\bf pointwise chaste} at $x$ if the analogous properties hold for the local intersection cohomology
groups at $x$.  If $X$ is pointwise chaste at $x$, we define 
$$P_{X,x}(t) := \sum_{i\geq 0} \dim \IH^{2i}_x(X; \Ql)\; t^i.$$

We say that $X$ is an {\bf affine cone} if it is affine and its coordinate ring $\Fq[X]$ admits
a non-negative grading with only scalars in degree zero.  The {\bf cone point} of $X$
is the closed point defined by the vanishing of all functions of positive degree.  If $X$
is an affine cone with cone point $x$, then $\IHX$ is canonically isomorphic to $\IH^*_x(X; \Ql)$ \cite[Corollary 1]{Springer-purity}.

\begin{proposition}\label{purity}
If $X$ is an affine cone of positive dimension, then $X$ is pure and 
$\IHiX = 0$ for all $i\geq \dim X$.
\end{proposition}

\begin{proof}
Let $U\subset X$ be the complement of the cone point, and let $Z = U/\mathbb{G}_m = \Proj \Fq[X]$.
Let $j:U\to X$ be the inclusion; then $\IC_X = j_{!*}\IC_U = \tau^{<0}Rj_*\IC_U$,
so $$\IHiX = \H^{i-\dim X}(X; \IC_X) = \H^{i-\dim X}(X; \tau^{<0}Rj_*\IC_U)$$ vanishes when $i\geq \dim X$,
and it is equal to $\IH^i(U; \Ql)$ when $i<\dim X$.

By the Leray-Serre spectral sequence applied to the $\mathbb{G}_m$-bundle $U\to Z$, 
combined with the hard Lefschetz theorem for $\IH^*(Z; \Ql)$, $\IH^i(U; \Ql)$ is isomorphic
to the space of primitive vectors in $\IH^*(Z; \Ql)$ for all $i<\dim X$.  Thus purity of $X$ follows from
purity of the projective variety $Z$.
\end{proof}

\begin{remark}
Proposition \ref{purity} is well-known to experts; in particular,
a version of the argument above can also be found in
\cite[4.2]{Brion-Joshua} and \cite[3.1]{dCM}.
\end{remark}

\excise{
Let $X = \sqcup X_\b$ be a stratification of $X$, and let $x\in X_\b$.  
We say that an affine cone $S$ with cone point $s$ is a {\bf normal slice} to $X_\b$ at $x$ if there exists
a vector space $V$, an \'etale neighborhood of $(0,s)\in V\tilde S$, and an isomorphism to an \'etale
neighborhood of $x\in X$ such that the induced map on tangent cones
$$V\times S \;\cong\; TC_{(0,s)} (V \times S)\to TC_x X$$
is an isomorphism that takes $V$ to $T_x X_\b$.

\begin{lemma}\label{slice}
If $S$ is a normal slice to $X_\b$ at $x$, then $\IH^*(S; \Ql) \cong \H^{*-\dim X}(\IC_{X,x})$.
\end{lemma}

\begin{proof}
A local \'etale isomorphism induces an isomorphism of \'etale cohomology groups of stalks
of intersection cohomology sheaves.  Since $S$ is a cone,
its local intersection cohomology at the cone point coincides with its global
intersection cohomology.
\end{proof}
}

The following combinatorial lemma will be needed in the proof of Theorem \ref{chastity};
the statement and proof of this lemma were communicated to us by Ben Webster.
Let $k$ be a field of characteristic zero.
For all positive integers $m,n,s$, consider the super power sum polynomial
$$p_{m,n,s}(x,y) := x_1^s + \cdots + x_m^s - y_1^s -\cdots - y_n^s,$$
where $x = (x_1,\ldots,x_m)\in k^m$ and 
$y = (y_1,\ldots,y_n)\in k^n$.

\begin{lemma}\label{super power sums}
Suppose that 
$p_{m,n,s}(x,y) = p_{m',n',s}(x',y')$ for all $s\geq 0$,
that $x_i\neq 0\neq y_i$ for all $i$,
and that $x_i\neq y_j$ and $x_i' \neq y_i'$ for all $i,j$.
Then $m=m'$, $n=n'$, and $(x,y)$ may be taken to $(x',y')$ by an element of $S_m\times S_n$.
\end{lemma}

\begin{proof}
Consider the rational function
$$f(z) := \frac{(1-x_1z)\cdots (1-x_mz)}{(1-y_1z)\cdots (1-y_nz)} \in k(z).$$
We have
\begin{eqnarray*}
f(z) &=& \exp\log f(z)\\\\
&=& \exp\Big(\log(1-x_1z)+\cdots+\log(1-x_mz)-\log(1-y_1z)-\cdots-\log(1-y_nz)\Big)\\\\
&=& \exp\sum_{s=1}^\infty\left(\frac{(x_1z)^s}{s} +\cdots +\frac{(x_mz)^s}{s}
- \frac{(y_1z)^s}{s}-\cdots-\frac{(y_nz)^s}{s}\right)\\\\
&=& \exp\sum_{s=1}^\infty \frac{p_{m,n,s}(x,y)}{s} z^{s}.
\end{eqnarray*}
A rational function over a field of characteristic zero is determined by its Taylor expansion
at zero, thus the values of the super power sums determine the rational function $f(z)$.
By looking at zeros and poles of $f(z)$ with multiplicity, they determine $m$, $n$, $x$ (up to permutation), 
and $y$ (up to permutation).
\end{proof}

We say that a variety $Y$ over $\Fq$ has {\bf polynomial count} if there exists a polynomial $\nu_Y(t)
\in\mathbb{Z}[t]$ such that,
for all $s\geq 1$, $|Y(\Fqs)| = \nu(q^s)$.  Let $X$ be an affine cone, and let $X = \bigsqcup X_\b$ be a stratification such that $X_0$ is the only zero-dimensional stratum, consisting only of the cone point.

\begin{theorem}\label{chastity}
Suppose that $X_\b$ has polynomial count for all $\b$, and that
$X\smallsetminus X_0$ is everywhere pointwise chaste
with local intersection cohomology Poincar\'e polynomial
$P_{X,x}(t) = P_\b(t)$ for all $x\in X_\b$.
Then $X$ is chaste (and therefore also pointwise chaste at the cone point), 
and $$t^{\dim X}P_X(t^{-1}) = \sum_\b \nu_{X_\b}(t) P_{\b}(t).$$
\end{theorem}

\begin{proof}
Consider the Frobenius automorphism $\Fr^*_{c}$ of $\IH^*_c(X; \Ql)$, the compactly supported
intersection cohomology group.
By Poincar\'e duality \cite[II.7.3]{KW}, we have
\begin{equation*}\label{pd}
q^{s\dim X}\operatorname{tr}\!\Big((\Fr^*)^{-s}\curvearrowright \IH^{2\dim X - i}(X; \Ql)\Big) = \operatorname{tr}\!\Big((\Fr_c^*)^s\curvearrowright \IH_c^i(X; \Ql)\Big).
\end{equation*}
By the Lefschetz formula \cite[III.12.1(4)]{KW}, we have
$$\sum_{i\geq 0}(-1)^i \operatorname{tr}\!\Big((\Fr_c^*)^s\curvearrowright \IH_c^i(X; \Ql)\Big)
= \sum_{x\in X(\Fqs)}\sum_{i\geq 0}(-1)^i \operatorname{tr}\!\Big((\Fr_x^*)^s\curvearrowright\IHxX\Big).$$
If $x\in X_\b(\Fqs)$ for some $\b\neq 0$, then $x$ contributes $P_{\b}(q^s)$ to this sum.  If $x$ is the cone point, then 
$\IHxX \cong \IHX$, so $x$ contributes $\sum(-1)^i\operatorname{tr}\!\big((\Fr^*)^s\curvearrowright \IH^i(X; \Ql)\big)$.
Thus we have
\begin{equation}\label{lef}
\sum_{i\geq 0}(-1)^i \left(q^{s\dim X} \operatorname{tr}\!\Big((\Fr^*)^{-s}\curvearrowright \IH^{2\dim X-i}(X; \Ql)\Big) - \operatorname{tr}\!\Big((\Fr^*)^{s}\curvearrowright \IH^i(X; \Ql)\Big)\right)
= \sum_{\b\neq 0}\nu_{X_\b}(q^s)P_{\b}(q^s).
\end{equation}

Let $r_i = \dim \IH^i(X; \Ql)$, and let $\a_{1,i},\cdots,\a_{r_i,i}$ be the eigenvalues of $\Fr^*\curvearrowright\IH^i(X; \Ql)$, counted with multiplicity.
Write $$\a_i := (\a_{i,1},\cdots,\a_{i,r_i}) \and q^{\dim X}\!/\a_i := (q^{\dim X}\!/\a_{i,1},\cdots,q^{\dim X}\!/\a_{i,r_i}),$$
so that the left-hand side
of Equation \eqref{lef} is equal to 
$$\sum_{i\geq 0}(-1)^i p_{r_i,r_i,s}(q^{\dim X}\!/\a,\a).$$
The right-hand side is a polynomial in $q^s$ with integer coefficients,
and therefore can be written in the form $p_{m,n,s}(x,y)$, where the entries of $x$ and $y$ 
are all non-negative powers of $q$ and the powers that appear in $x$ 
are distinct from the powers that appear in $y$.  
Assuming that $\dim X>0$, Proposition \ref{purity} tells us that the entries of $\a_i$ are disjoint from the entries of $q^{\dim X}\!/\a_i$,
thus the hypotheses of Lemma \ref{super power sums}
are satisfied and each $\a_{i,j}$ is equal to a power of $q$.  
Since we already know that $X$ is pure, this 
implies that $X$ is chaste, and Equation \eqref{lef} becomes 
$$q^{s\dim X}P_X(q^{-s}) - P_X(q^s) = \sum_{\b\neq 0}\nu_{X_\b}(q^s)P_{\b}(q^s).$$
Since this holds for all $s$, we may replace $q^s$ with the variable $t$.
Moving $P_X(t) = P_{0}(t)$ to the right-hand side, and noting that $\nu_{X_0}(t) = 1$,
we obtain the desired equality.
\end{proof}

\subsection{Cohomological interpretation of Kazhdan-Lusztig polynomials}
We now combine the results of Subsections \ref{sec:xa}, \ref{sec:local}, 
and \ref{sec:ih} to give a cohomological interpretation
of the polynomial $P_{M_\cA}(t)$.  Let $\cA$ be an arrangement over a finite field $\Fq$, and let $\ell$ be a prime
that does not divide $q$.

\begin{lemma}\label{slice-OT}
For every flat $F$ and every element $x\in X_{\cA,F}$, 
$\IH^*_x(X_\cA; \Ql) \cong \IH^*(X_{\!\cA^F}; \Ql)$.
\end{lemma}

\begin{proof}
By Theorem \ref{etale}, we have an \'etale map from a
neighborhood of $x\in X_\cA$ to a neighborhood of $(0,x) \in X_{\!\cA^F}\times X_{\cA,F}$.
It follows that the local intersection cohomology of $X_\cA$
at $x$ is isomorphic to the local intersection cohomology of $X_{\!\cA^F}$ at the cone point
times the local intersection cohomology of $X_{\cA,F}$ at $x$.
By the contraction lemma \cite[Corollary 1]{Springer-purity}, 
the local intersection cohomology of $X_{\!\cA^F}$ at the cone point is isomorphic to the 
the global intersection cohomology of $X_{\!\cA^F}$.
Since $X_{\cA,F}$ is smooth, the local intersection
cohomology of $X_{\cA,F}$ is trivial.
\end{proof}

Let $\chi_\cA(t) = \chi_{M_\cA}(t)$ be the characteristic polynomial of $\cA$.
The variety $U_\cA$ is polynomial count with $\nu_{U_\cA}(t) = \chi_\cA(t)$ \cite[2.69]{OT}.
For any arrangement $\cA$ in $V$, let $$\rk\cA := \rk M_\cA = \dim V = \dim X_\cA.$$

\begin{proposition}\label{arithmetic recursion}
The reciprocal plane $X_\cA$ is chaste, and 
$$t^{\rk\cA}P_{X_\cA}(t^{-1}) = \sum_F\chi_{\cA_F}(t)P_{X_{\!\cA^F}}(t).$$
\end{proposition}

\begin{proof}
We proceed by induction on the rank of $\cA$.  If $\rk\cA = 0$, the statement is trivial.
Now assume that the proposition holds for all arrangements of smaller rank.  In particular, this means that $X_{\!\cA^F}$ is chaste for all nonempty flats $F$.  By Lemma \ref{slice-OT}, this implies that $X_\cA$
is pointwise chaste away from the cone point, with $P_{X_\cA,x}(t) = P_{X_{\!\cA^F}}(t)$
for all $F$ nonempty and $x\in X_{\cA,F}$.  The statement then 
follows from Theorem \ref{chastity}.
\end{proof}

As a consequence, we find that the intersection cohomology Poincar\'e polynomial of a reciprocal plane over a finite field
coincides with the Kazhdan-Lusztig polynomial of the corresponding matroid.

\begin{theorem}\label{comb=geom}
If $\cA$ is an arrangement over a finite field, then $P_{X_\cA}(t) = P_{M_\cA}(t)$.
\end{theorem}

\begin{proof}
This follows from Proposition \ref{purity}, Theorem \ref{arithmetic recursion}, and the uniqueness 
of Theorem \ref{KL-exists}.
\end{proof}

\begin{corollary}\label{nonneg-cor}
If a matroid $M$ is representable, then $P_M(t)$ has non-negative coefficients.
\end{corollary}

\begin{proof}
If $M$ is representable over some field, then it is representable over a finite field \cite[Theorems 4 \& 6]{Rado},
and the corollary follows from Theorem \ref{comb=geom}.
\end{proof}

Let $\cA$ be an arrangement over $\C$.  Theorem \ref{comb=geom} says that we may interpret $P_{M_\cA}(t)$ geometrically
by choosing a representation of $M_\cA$ over a finite field and considering the $\ell$-adic \'etale intersection cohomology 
of the resulting reciprocal plane.  However, one might prefer to think about the topological intersection cohomology groups of 
$X_\cA(\C)$.  Let $P_{X_\cA}(t) = \sum_{i\geq 0} \dim \IH^{2i}(X_\cA(\C); \C)\; t^i$.

\begin{proposition}\label{topological}
If $\cA$ is an arrangement over $\C$, then
the topological intersection cohomology of $X_\cA(\C)$ vanishes in odd degree, and 
$P_{X_\cA}(t) = P_{M_\cA}(t)$.  Furthermore, the topological analogue of Lemma \ref{slice-OT} holds.
\end{proposition}

\begin{proof}
Choose a spreading out of $X_\cA$ and then base change to a finite field $\Fq$ of sufficiently large characteristic.
The fact that the topological intersection cohomology of $X_\cA(\mathbb{C})$
coincides with the graded dimension of the $\ell$-adic \'etale intersection cohomology of $X_\cA(\Fqb)$
after tensoring with $\C$ follows from \cite[6.1.9]{BBD} (see also \cite[1.4.8.1]{Conrad-etale}).
The same goes for local intersection cohomology groups.
\end{proof}

\begin{remark}
For $\cA$ an arrangement over a finite field or $\C$,
the isomorphism class of the variety $X_\cA$ is not determined by the matroid $M_\cA$.
However, Theorem \ref{comb=geom} and Proposition \ref{topological} imply 
that the intersection cohomology Poincar\'e polynomial $P_{X_\cA}(t)$ is determined by $M_\cA$.
\end{remark}

\subsection{Relation to the work of Li and Yong}
Li and Yong \cite{LiYong}
associate to any variety $Y$ over a field $k$ and any closed point $p\in Y$ two polynomials:  
$$P_{p,Y}(t) := \sum_{i\geq 0} \dim \H^{2i-\dim Y}\!(Y; \IC_{Y,p})\, t^i\and
H_{p,Y}(t) := (1-t)^{\dim Y}\!\operatorname{Hilb}(k[TC_pY]; t).$$
If $Y$ is a Schubert variety, then $P_{p,Y}(t)$ is an ordinary Kazhdan-Lusztig polynomial.
If the Schubert variety is covexillary, they prove that $\deg P_{p,Y}(t) = \deg H_{p,Y}(t)$, and that the coefficients of $H_{p,Y}(t)$
are greater than or equal to the corresponding coefficients of $P_{p,Y}(t)$ \cite[1.2]{LiYong}.  They conclude by asking for what
pairs $(p,Y)$ this same statement holds \cite[7.1]{LiYong}.

If $\cA$ is an arrangement over $k=\Fq$ or $\C$, $Y = X_\cA$, and $p\in X_{\cA,F}$, then Lemma \ref{slice-OT}
and Theorem \ref{comb=geom} (if $k=\Fq$) or Proposition \ref{topological} (if $k=\C$) tell us that $$P_{p,Y}(t) = P_{M_{\!\cA^F}}(t).$$  
Furthermore, Lemma \ref{slice-OT}
and \cite[4.3]{Berget}\footnote{The Hilbert series of the Orlik-Terao algebra in characteristic zero
was computed in \cite[1.2]{Terao-OT} and independently in 
\cite[Proposition 7]{PS}; Berget's proof is the first one that works in positive characteristic.}
tell us that $$H_{p,Y}(t) = h^{\operatorname{bc}}_{M_{\!\cA^F}}(t),$$ the $h$-polynomial
of the broken circuit complex of $M_{\!\cA^F}$.

Both properties studied by Li and Yong fail in general for $X_\cA$;
for example, if $M_{\!\cA^F}$ is the uniform matroid
of rank $d$ on a set of cardinality $d+1$, we have
$$h^{\operatorname{bc}}_{M_{\!\cA^F}}(t) = 1 + t + t^2 + \cdots + t^{d-1},$$
while $P_{M_{\!\cA^F}}(t)$ is a polynomial of degree less than $\frac{d}{2}$ with linear coefficient 
equal to $\binom{d+1}{2} - (d+1)$ (Corollary \ref{uniform coefficients}).  It would be interesting to determine
whether there is a nice class of ``covexillary matroids" for which $h^{\operatorname{bc}}_{M}(t)$ dominates $P_M(t)$.

\section{Algebra}\label{sec:algebra}
In this section we define a $q$-deformation of the M\"obius algebra of a matroid, use Kazhdan-Lusztig polynomials
to define a special basis for this algebra, and conjecture that the structure coefficients for this basis are non-negative.
We then verify the conjecture for Boolean matroids, and for uniform matroids and braid matroids of rank at most 3.

\subsection{The deformed M\"obius algebra}
Fix a matroid $M$.  The {\bf M\"obius algebra} is defined to be the free abelian group
$$E(M) := \Z\{\ep_F\mid F\in L(M)\}$$ equipped with the multiplication $\ep_F\cdot \ep_G := \ep_{F\vee G}$.
We define a deformation
$$E_q(M) := \Z[q,q^{-1}]\{\ep_F\mid F\in L(M)\}$$
with multiplication
$$\ep_F\cdot \ep_G \,:=\, \sum_{H\geq I\geq F\vee G} \mu(I,H)\, q^{\crk I}\; \ep_H,$$
where $\crk I := \rk M - \rk I$ is the corank of $I$.
The fact that we recover our original multiplication when $q=1$ follows from the fact 
that $\sum_{H\geq I\geq F\vee G} \mu(I,H) = \delta(H,F\vee G)$.

\begin{proposition}\label{deformation}
The $\Z[q,q^{-1}]$-algebra $E_q(M)$ is commutative, associative, and unital, with unit
equal to 
$$\displaystyle\sum_{F\leq G}\mu(F,G)\, q^{-\crk F}\; \ep_G.$$
\end{proposition}

\begin{proof}
Commutativity is immediate from the definition.  For associativity, we note that
\begin{eqnarray*}
\ep_F\cdot \ep_G &=&\sum_{H\geq I\geq F\vee G} \mu(I,H)\, q^{\crk I}\; \ep_H\\\\
&=& \sum_{H,I} \zeta(F,I)\zeta(G,I)\mu(I,H)\, q^{\crk I}\; \ep_H,
\end{eqnarray*}
and therefore
\begin{eqnarray*}
(\ep_F\cdot \ep_G) \cdot \ep_J
&=& \sum_{H,I,K,L} \zeta(F,I)\zeta(G,I)\zeta(H,L)\zeta(J,L)\mu(I,H)\mu(L,K)\, q^{\crk I+\crk L}\; \ep_K\\\\
&=& \sum_{I,K,L}\zeta(F,I)\zeta(G,I)\zeta(J,L)\mu(L,K)\, q^{\crk I+\crk L}\sum_H \mu(I,H)\zeta(H,L)\; \ep_K\\\\
&=& \sum_{I,K,L}\zeta(F,I)\zeta(G,I)\zeta(J,L)\mu(L,K)\, q^{\crk I+\crk L}\delta(I,L)\; \ep_K\\\\
&=& \sum_{I,K}\zeta(F,I)\zeta(G,I)\zeta(J,I)\mu(I,K)\, q^{2\crk I}\;\ep_K.
\end{eqnarray*}
This expression is clearly symmetric in $F$, $G$, and $J$, hence our product is associative.

For the statement about the unit, we observe that
\begin{eqnarray*}
\left(\sum_{F\leq G}\mu(F,G)\, q^{-\crk F}\; \ep_G\right)\cdot \ep_H &=&
\sum_{F\leq G}\mu(F,G)\, q^{-\crk F}\!\!\sum_{I\geq J\geq G\vee H}\mu(J,I)\, q^{\crk J}\;\ep_I\\\\
&=&\sum_{F,G,I,J}\mu(F,G)\zeta(G,J)\zeta(H,J)\mu(J,I)\, q^{\crk J-\crk F}\;\ep_I\\\\
&=&\sum_{F,I,J}\zeta(H,J)\mu(J,I)\, q^{\crk J-\crk F}\left(\sum_G\mu(F,G)\zeta(G,J)\right)\;\ep_I\\\\
&=&\sum_{F,I,J}\zeta(H,J)\mu(J,I)\, q^{\crk J-\crk F}\delta(F,J)\;\ep_I\\\\
&=&\sum_{I,J}\zeta(H,J)\mu(J,I)\;\ep_I\\\\
&=&\sum_{I}\delta(H,I)\;\ep_I\\\\
&=&\ep_H.
\end{eqnarray*}
This completes the proof.
\end{proof}

\subsection{The Kazhdan-Lusztig basis}
We now define a new basis for $E_q(M)$ in terms of the standard basis,
using Kazhdan-Lusztig polynomials to define the matrix coefficients.  The definition is analogous
to that of the Kazhdan-Lusztig basis for the Hecke algebra, and we therefore call our new basis
the {\bf Kazhdan-Lusztig basis}.

For all $F\in L(M)$, let
$$x_F := \sum_{G\geq F} q^{\rk G-\rk F}P_{M_G^F}(q^{-2})\;\ep_G.$$
It is clear that $$x_F\in \ep_F + q\Z[q]\{\ep_G\mid G> F\},$$
and therefore that $\{x_F\mid F\in L(M)\}$ is a $\Z[q,q^{-1}]$-basis for $E_q(M)$.
Even better, it is a $\Z[q]$-basis for the (non-unital) subring $\Z[q]\{\ep_F\mid F\in L(M)\}\subset E_q(M)$.

Consider the structure constants for multiplication in this basis.  That is, for all $F,G,H$,
define $C_{FG}^H(q) \in \Z[q]$ by the equation
$$x_F\cdot x_G = \sum_H C_{FG}^H(q)\; x_H.$$
We conjecture that this polynomial has non-negative coefficients.

\begin{conjecture}\label{structure coefs}
For all $F,G,H\in L(M)$, $C_{FG}^H(q) \in \N[q]$.
\end{conjecture}

\begin{subsection}{Boolean matroids}
In this subsection we will prove Conjecture \ref{structure coefs} for Boolean matroids
by producing an explicit formula for multiplication in the Kazhdan-Lusztig basis.  We first need the following
two lemmas.

\begin{lemma}\label{venn}
Fix subsets $F,G,L\subset [n]$ with $F\cup G\subset L$.  Let $F\Delta G := F\cup G\smallsetminus F\cap G$
be the symmetric difference of $F$ and $G$.  Then
$$\sum_{\substack{H\supset F\\ I\supset G\\ H\cup I=L}}q^{|H|+|I|-|F|-|G|} = (1+q)^{|F\Delta G|}
(2q+q^2)^{|L\smallsetminus F\cup G|}.$$
\end{lemma}

\begin{proof}
The trick is to write $$H = F\sqcup H'\sqcup H''\sqcup J
\and I = G\sqcup I'\sqcup I''\sqcup J,$$
where
$$H'= F^c\cap G \cap H,\;
H'' = F^c\cap G^c\cap H\cap I^c,\;
I'= F\cap G^c\cap I,\;
I'' = F^c\cap G^c\cap H^c \cap I,\;\text{and}\;
J = F^c\cap G^c\cap H\cap I.$$
Then the left-hand side becomes
$$\sum_{H',H'',I',I'',J}q^{|H'|+|H''|+|I'|+|I''|+2|J|},$$
where the sum is over $H'\subset F^c\cap G$, $I'\subset F\cap G^c$, and $H'',I'',J\subset F^c\cap G^c\cap L$
with $$H''\sqcup I''\sqcup J = F^c\cap G^c\cap L.$$
We have $$\sum_{H',I'}q^{|H'|+|I'|} = (1+q)^{|F\Delta G|},$$
and, for each fixed $J$,
$$\sum_{H'',I''}q^{|H''|+|I''|} 
= (2q)^{|F^c\cap G^c\cap L\cap J^c|}.$$
Thus the left-hand side is equal to
$$(1+q)^{|F\Delta G|}\sum_{J\subset F^c\cap G^c\cap L}q^{2|J|}(2q)^{|F^c\cap G^c\cap L\cap J^c|}
= (1+q)^{|F\Delta G|}(2q+q^2)^{|F^c\cap G^c\cap L|},$$
where the last equality is an application of the binomial theorem.
\end{proof}

\begin{lemma}\label{binomial}
Fix subsets $F\subset G\subset [n]$.
Then for any polynomials $f(q)$ and $g(q)$, we have
$$\sum_{F\subset H \subset G}f(q)^{|G\smallsetminus H|}g(q)^{|H\smallsetminus F|} 
= \Big(f(q) + g(q)\Big)^{|G\smallsetminus F|}.$$
\end{lemma}

\begin{proof}
This is simply a reformulation of the binomial theorem.
\end{proof}

\begin{proposition}\label{Boolean}
Let $M$ be the Boolean matroid on the ground set $[n]$.
Then for any subsets $F,G\subset [n]$, we have 
$$x_F\cdot x_G = \sum_{K\supset F\cup G} q^{n-|K|} (1+q)^{|K| - |F\cap G|}\, x_K.$$
\end{proposition}

\begin{proof}
For each $F\subset G$, $M_G^F$ is again Boolean, so $P_{M_G^F}(t) = 1$ by Corollary \ref{boolean}.
This means that $$x_F = \sum_{G\supset F} q^{|G\smallsetminus F|}\; \ep_G,$$
and, by M\"obius inversion,
$$\ep_F = \sum_{G\supset F} (-q)^{|G\smallsetminus F|}\; x_G.$$
We therefore have
\begin{eqnarray*}
x_F\cdot x_G &=& \left(\sum_{H\supset F} q^{|H\smallsetminus F|} \ep_H\right)
\cdot \left(\sum_{I\supset G} q^{|I\smallsetminus G|} \ep_I\right)\\\\
&=& \sum_{\substack{H\supset F\\ I\supset G}}q^{|H|+|I|-|F|-|G|}\;\ep_H\cdot \ep_I\\\\
&=& \sum_{\substack{H\supset F\\ I\supset G}}q^{|H|+|I|-|F|-|G|} 
\sum_{J\supset H\cup I} q^{n-|J|}(1-q)^{|J\smallsetminus H\cup I|}\;\ep_J\\\\
&=& \sum_{\substack{H\supset F\\ I\supset G}}q^{|H|+|I|-|F|-|G|} 
\sum_{J\supset H\cup I} q^{n-|J|}(1-q)^{|J\smallsetminus H\cup I|}
\sum_{K\supset J} (-q)^{|K\smallsetminus J|}\; x_K.
\end{eqnarray*}

By Lemma \ref{venn}, this equation becomes
$$x_F\cdot x_G = \sum_{K\supset J\supset L\supset F\cup G} (1+q)^{|F\Delta G|}
(2q+q^2)^{|L\smallsetminus F\cup G|}q^{n-|J|}(1-q)^{|J\smallsetminus L|}
(-q)^{|K\smallsetminus J|}\; x_K.$$
By writing $n-|J| = n-|K| + |K\smallsetminus J|$, we may rewrite our equation as
$$x_F\cdot x_G = (1+q)^{|F\Delta G|}\sum_{K\supset J\supset L\supset F\cup G}
q^{n-|K|}(2q+q^2)^{|L\smallsetminus F\cup G|}(1-q)^{|J\smallsetminus L|}
(-q^2)^{|K\smallsetminus J|}\; x_K.$$
Applying Lemma \ref{binomial} first to the sum over $J$ and then to the sum over $L$, this becomes
\begin{eqnarray*}
x_F\cdot x_G &=& (1+q)^{|F\Delta G|}\sum_{K\supset L\supset F\cup G}
q^{n-|K|}(2q+q^2)^{|L\smallsetminus F\cup G|}(1-q-q^2)^{|K\smallsetminus L|}\; x_K\\\\
&=& (1+q)^{|F\Delta G|} \sum_{K\supset F\cup G}q^{n-|K|}(1+q)^{|K\smallsetminus F\cup G|}\; x_K\\\\
&=& \sum_{K\supset F\cup G}q^{n-|K|}(1+q)^{|K| - |F\cap G|}\; x_K.
\end{eqnarray*}
This completes the proof.
\end{proof}
\end{subsection}

\subsection{Uniform matroids}\label{sec:alg:uni}
In this subsection we give the multiplication table for $E_q(M)$ in terms of the Kazhdan-Lusztig basis
when $M$ is a uniform matroid of rank at most 3.  The rank 1 case is covered by
Proposition \ref{Boolean} with $n=1$.

\begin{example}
Let $M$ be the uniform matroid of rank 2 on the ground set $[n] = \{1,\ldots,n\}$.
In this case, $P_{M^F_G}(t) = 1$ for all $F\leq G$ (since $\rk M = 2$), and we have
the following multiplication table:
\begin{eqnarray*}
x_{[n]}^2 &=& x_{[n]}\\
x_{[n]}\cdot x_{\{i\}} &=& (1+q)x_{[n]}\\
x_{[n]}\cdot x_{\emptyset} &=& (1+nq+q^2)x_{[n]}\\
x_{\{i\}}^2 &=& qx_{\{i\}} + (1+q)x_{[n]}\\
x_{\{i\}}\cdot x_{\{j\}} &=& (1+q)^2x_{[n]}\qquad (i\neq j)\\
x_{\{i\}}\cdot x_\emptyset &=& q(1+q)x_{\{1\}} + \Big(1+ nq + (n-1)q^2 \Big) x_{[n]}\\
x_\emptyset^2 &=& q^2 x_\emptyset + q(1+q)\textstyle\sum_{i}x_{\{i\}} + \Big(1+nq+(n-1)^2q^2\Big)x_{[n]}.
\end{eqnarray*}
\end{example}

\begin{example}\label{uniform rank 3}
Let $M$ be the uniform matroid of rank 3 on the ground set $[n]$.  In this case, 
$P_{M^F_G}(t) = 1$ for all $F\leq G$ unless $F = \emptyset$ and $G = [n]$,
in which case Corollary \ref{uniform coefficients} tells us that 
$$P_{M^F_G}(t) = P_M(t) = 1 + \Big(\tbinom{n}{2} - n\Big)t.$$
We have the following multiplication table:
\begin{eqnarray*}
x_{[n]}^2 &=& x_{[n]}\\
x_{[n]} \cdot x_{\{i,j\}} &=& (1+q) x_{[n]}\\
x_{[n]} \cdot x_i &=& \Big(1+(n-1)q+q^2\Big) x_{[n]}\\
x_{[n]} \cdot x_{\emptyset} &=& \Big(1 + \tbinom{n}{2}q + \tbinom{n}{2}q^2 + q^3\Big) x_{[n]}\\
x_{\{i,j\}} \cdot x_{\{i,j\}} &=& q x_{\{i,j\}} + (1+q) x_{[n]}\\
x_{\{i,j\}} \cdot x_{\{i,k\}} &=& (1+q)^2 x_{[n]}\qquad (j\neq k)\\
x_{\{i,j\}}\cdot x_{\{k,\ell\}} &=& (1+q)^2 x_{[n]}\qquad (\{i,j\}\cap\{k,\ell\}=\emptyset)\\
x_{\{i,j\}}\cdot x_{\{i\}} &=& q(1+q) x_{\{i,j\}} + \Big(1 + (n-1)q + (n-2)q^2\Big) x_{[n]}\\
x_{\{i,j\}}\cdot x_k &=& (1 + nq + nq^2 + q^3) x_{[n]}\qquad (k\notin\{i,j\})\\
x_{\{i,j\}}\cdot x_{\emptyset} &=& q(1+q)^2 x_{\{i,j\}} + \Big(1 + \tbinom{n}{2}q + (n^2 - n - 3) q^2 
+ \Big(\tbinom{n}{2}-2\Big) q^3\Big) x_{[n]}\\
x_{\{i\}}^2 &=& q^2 x_{\{i\}} + q(1+q) \textstyle\sum_{j\neq i} x_{\{i,j\}} + \Big(1 + (n-1)q + (n-2)^2q^2\Big) x_{[n]}\\
 x_{\{i\}} x_{\{j\}} &=& q(1+q)^2 x_{\{i,j\}} + \Big(1 + (2n-3)q + n(n-2)q^2 + (2n-5)q^3\Big) x_{[n]}\qquad (i\neq j)\\
x_{\{i\}}\cdot x_{\emptyset} &=& q^2(1+q) x_i + q(1+q)^2 \textstyle\sum_{j\neq i} x_{\{i,j\}}\\
&& + \Big(1 + \tbinom{n}{2}q + \tfrac{1}{2}(n-1)(n^2-6)q^2
+ \tfrac{1}{2}(n-1)(n^2-8)q^3 + \tfrac{1}{2}n(n-3)q^4\Big) x_{[n]}\\
x_{\emptyset}^2 &=& q^3 x_{\emptyset} + q^2(1+q)\textstyle\sum_i x_{\{i\}} + q(1+q)^2 \sum_{i<j} x_{\{i,j\}}\\
&& + \Big(1+\tbinom{n}{2}q + \tfrac{1}{4}n(n^3 - 2n^2 - n -2)q^2
+ \tfrac{1}{2}(n-1)(n^3 - n^2 -5n - 2) q^3\\
&&\qquad + \tfrac{1}{4}n^2(n+1)(n-3)q^4 + \tfrac{1}{2}n(n-3)q^5\Big)x_{[n]}.
\end{eqnarray*}
Note that each coefficient is non-negative for all $n\geq 3$, and when $n=3$, 
this multiplication table agrees with Proposition \ref{Boolean}.
\end{example}

\begin{remark} \label{needKLornotpositive}
It is reasonable to ask if Conjecture \ref{structure coefs} would still hold
if we were to redefine the Kazhdan-Lusztig basis by putting
$x_F := \sum_{G\geq F} q^{\rk G-\rk F}\;\ep_G$; that is, if we were to pretend that all
Kazhdan-Lusztig polynomials were equal to 1.  If we did this, 
the linear term of $C_{\emptyset\emptyset}^{[n]}(q)$
in Example \ref{uniform rank 3} would be equal to $2n - \binom{n}{2}$, which is negative when $n>5$.
Thus the Kazhdan-Lusztig polynomials truly play a necessary role in Conjecture \ref{structure coefs}.
\end{remark}


\subsection{Braid matroids}\label{sec:alg:braid}
Finally, we consider braid matroids of small rank.  The braid matroid $M_2$ is isomorphic to the
Boolean matroid of rank 1.  The braid matroid $M_3$ is isomorphic to the uniform matroid of rank 2
on the ground set $[3]$.

\begin{example}\label{braid 4}
Let $M_4$ be the braid arrangement of rank 3.  
The ground set $\cI$ has cardinality $\tbinom{4}{2} = 6$, and $P_{M_4}(t) = 1 + t$ (see appendix).
Flats correspond to set theoretic partitions of $[4]$.  Flats of rank 1 all have cardinality 1,
corresponding to partitions of $[4]$ into one set of cardinality 2 and two singletons.
Flats of rank 2 come in two types:  those of cardinality 2 (partitions of $[4]$ into
two subsets of cardinality 2), and those of cardinality 3 (partitions of $[4]$ into one subset
of cardinality 3 and one singleton).
We omit the full multiplication table, but give the single most interesting product:
\begin{eqnarray*}
x_{\emptyset}^2 &=& q^3 x_{\emptyset} + q^2(1+q)\sum_{|F|=1}x_F
+ q(1+q)^2\sum_{|F|=2}x_F + q(1+3q+4q^2)\sum_{|F|=3}x_F\\\\
&& + \big(q^5 + 13q^4 + 38q^3 + 32q^2 + 7q + 1\big) x_\cI.
\end{eqnarray*}
\end{example}

\subsection{Update}\label{update}
After this paper was published, Ben Young wrote software to compute the polynomials
$C_{FG}^H(q)$ in SAGE.  Using this software, we discovered that 
that Conjecture \ref{structure coefs} is false for a number of examples in rank 4 and higher,
including the uniform matroid of rank 4 on 6 elements, the uniform matroid of rank 6 on 7 elements,
and the supersolvable matroid represented by all vectors in $\mathbb{F}_2^4$ with at least two coordinates
equal to 1.  We were not able to find any braid matroids for which the conjecture fails (we computed all of
the structure coefficients for $M_5$, $M_6$, and $M_7$, and some of the structure coefficients for $M_8$).

In addition, we used this software to identify and correct minor errors
in Examples \ref{uniform rank 3} and \ref{braid 4}.  In Example \ref{uniform rank 3}, we corrected
the coefficient of $q^4$ in $C_{\emptyset\emptyset}^{[n]}(q)$.
In Example \ref{braid 4}, we corrected the coeffiecients of $q$, $q^2$, and $q^3$ in $C_{\emptyset\emptyset}^\cI(q)$.

\appendix
\section{Appendix (with Ben Young)}
We include here computer generated computations of Kazhdan-Lusztig polynomials of uniform matroids and braid
matroids of small rank.  Individual Kazhdan-Lusztig polynomials are to be read vertically; for example, 
Table \ref{1d} tells us that the Kazhdan-Lusztig polynomial
of $M_{1,8}$ is equal to $1+27t+120t^2+84t^3$.

We see some interesting patterns in the tables.  First, we find experimental evidence for Conjecture \ref{log concave}.
Also, with the help of the On-Line Encyclopedia of Integer Sequences \cite{oeis}, we can find formulas for specific coefficients.
For example, we observe
that the leading coefficient of the Kazhdan-Lusztig polynomial of the uniform matroid $M_{1,2k-1}$
is equal to the Catalan number $C_k = \frac{1}{k+1}\binom{2k}{k}$, and the
leading coefficient of the Kazhdan-Lusztig polynomial of the braid matroid $M_{2k}$
is equal to $(2k-3)!!\, (2k-1)^{(k-2)}$.
The former statement, along with a combinatorial description of all coefficients of Kazhdan-Lusztig polynomials
of uniform matroids, is proved in \cite{GPY}.

The sage code which was used to compute these tables is available at \url{https://github.com/benyoung/kl-matroids}.

\excise{
\begin{conjecture}
The leading coefficient of the Kazhdan-Lusztig polynomial of the uniform matroid $M_{1,2k-1}$
is equal to the Catalan number $C_k = \frac{1}{k+1}\binom{2k}{k}$.
\end{conjecture}

\begin{conjecture}
The leading coefficient of the Kazhdan-Lusztig polynomial of the uniform matroid $M_{1,2k}$
is equal to the binomial coefficient $\binom{2k+1}{k-1}$.  This is also equal to the leading
coefficient of the Kazhdan-Lusztig polynomial of $M_{2,2k-1}$.
\end{conjecture}

\begin{conjecture}
The leading coefficient of the Kazhdan-Lusztig polynomial of the braid matroid $M_{2k}$
is equal to\footnote{See sequence number A034941 in the On-Line Encyclopedia
of Integer Sequences.} $(2k-3)!!\, (2k-1)^{(k-2)}$.
\end{conjecture}
}


\newpage
\subsection{Uniform matroids}

\begin{table}[h]\label{1d}
\caption{Kazhdan-Lusztig polynomials for the uniform matroid $M_{1,d}$}
\begin{tabular}{|c||r|r|r|r|r|r|r|r|r|r|r|r|r|r|r|}
\hline
$d=$&1&2&3&4&5&6&7&8&9&10&11&12&13&14&15\\
\hline\hline
1&1&1&1&1&1&1&1&1&1&1&1&1&1&1&1\\
\hline
$t$&&&2&5&9&14&20&27&35&44&54&65&77&90&104\\
\hline
$t^2$&&&&&5&21&56&120&225&385&616&936&1365&1925&2640\\
\hline
$t^3$&&&&&&&14&84&300&825&1925&4004&7644&13650&23100\\
\hline
$t^4$&&&&&&&&&42&330&1485&5005&14014&34398&76440\\
\hline
$t^5$&&&&&&&&&&&132&1287&7007&28028&91728\\
\hline
$t^6$&&&&&&&&&&&&&429&5005&32032\\
\hline
$t^7$&&&&&&&&&&&&&&&1430\\
\hline
\end{tabular}
\end{table}

\begin{table}[h]\label{2d}
\caption{Kazhdan-Lusztig polynomials for the uniform matroid $M_{2,d}$}
\begin{tabular}{|c||r|r|r|r|r|r|r|r|r|r|r|r|r|}
\hline
$d=$&3&4&5&6&7&8&9&10&11&12&13&14&15\\
\hline\hline
1&1&1&1&1&1&1&1&1&1&1&1&1&1\\
\hline
$t$&5&14&28&48&75&110&154&208&273&350&440&544&663\\
\hline
$t^2$&&&21&98&288&675&1375&2541&4368&7098&11025&16500&23936\\
\hline
$t^3$&&&&&84&552&2145&6380&16016&35672&72618&137760&246840\\
\hline
$t^4$&&&&&&&330&2805&13585&49049&146510&382200&899640\\
\hline
$t^5$&&&&&&&&&1287&13442&78078&331968&1150968\\
\hline
$t^6$&&&&&&&&&&&5005&62062&420784\\
\hline
$t^7$&&&&&&&&&&&&&19448\\
\hline
\end{tabular}
\end{table}

\begin{table}[h]\label{3d}
\caption{Kazhdan-Lusztig polynomials for the uniform matroid $M_{3,d}$}
\begin{tabular}{|c||r|r|r|r|r|r|r|r|r|r|r|r|}
\hline
$d=$&3&4&5&6&7&8&9&10&11&12&13&14\\
\hline\hline
1&1&1&1&1&1&1&1&1&1&1&1&1\\
\hline
$t$&9&28&62&117&200&319&483&702&987&1350&1804&2363\\
\hline
$t^2$&&&56&288&927&2365&5214&10374&19110&33138&54720&86768\\
\hline
$t^3$&&&&&300&2145&9020&28886&77714&184730&399840&803760\\
\hline
$t^4$&&&&&&&1485&13585&70499&271635&862680&2384760\\
\hline
$t^5$&&&&&&&&&7007&78078&482118&2171988\\
\hline
$t^6$&&&&&&&&&&&32032&420784\\
\hline
\end{tabular}
\end{table}

\newpage
\subsection{Braid matroids}


\begin{table}[h]\label{3-13}
\caption{Kazhdan-Lusztig polynomials for the braid matroid $M_n$}
\begin{tabular}{|c||r|r|r|r|r|r|r|r|r|r|r|r|r|}
\hline
$n=$&1&2&3&4&5&6&7&8&9&10&11&12&13\\
\hline\hline
1&1&1&1&1&1&1&1&1&1&1&1&1&1\\
\hline
$t$&&&&1&5&16&42&99&219&466&968&1981&4017\\
\hline
$t^2$&&&&&&15&175&1225&6769&32830&147466&632434&2637206\\
\hline
$t^3$&&&&&&&&735&16065&204400&2001230&16813720&128172330\\
\hline
$t^4$&&&&&&&&&&76545&2747745&56143395&864418555\\
\hline
$t^5$&&&&&&&&&&&&13835745&746080335\\
\hline
\end{tabular}
\end{table}

\begin{table}[h]\label{14-17}
\begin{tabular}{|c||r|r|r|r|r|r|r|r|r|r|r|r|r|r|r|}
\hline
$n=$&14&15&16&17\\
\hline\hline
1&1&1&1&1\\
\hline
$t$&8100&16278&32647&65399\\
\hline
$t^2$&10811801&43876001&176981207&711347303\\
\hline
$t^3$&915590676&6252966720&41362602281&267347356003\\
\hline
$t^4$&11200444255&129344350135&1377269949055&13819966094935\\
\hline
$t^5$&22495833360&502627875750&9305666915545&151395489770525\\
\hline
$t^6$&3859590735&293349030975&12290930276625&376566883537845\\
\hline
$t^7$&&&1539272109375&157277996100225\\
\hline
\end{tabular}
\end{table}

\begin{table}[h]\label{18-20}
\begin{tabular}{|c||r|r|r|}
\hline
$n=$&18&19&20\\
\hline\hline
1&1&1&1\\
\hline
$t$&130918&261972&524097\\
\hline
$t^2$&2853229952&11430715476&45762931992\\
\hline
$t^3$&1698735206324&10656703437054&66208557177786\\
\hline
$t^4$&132618161185510&1229703907984734&11100857399288280\\
\hline
$t^5$&2242336712846230&30941776173508200&404180066561961690\\
\hline
$t^6$&9443716601138820&205809448675350520&4042252614171772000\\
\hline
$t^7$&8758018896026400&352844128436870070&11522756204094885750\\
\hline
$t^8$&831766748637825&110176255068905025&7879824460254822075\\
\hline
$t^9$&&&585243816844111425\\
\hline
\end{tabular}
\end{table}

\bibliography{./symplectic}
\bibliographystyle{amsalpha}

\end{document}